\documentclass[reqno]{amsart}

\usepackage[pagebackref]{hyperref}
\usepackage{amsmath,amssymb,amsthm,mathrsfs,cite}
\theoremstyle{plain}

\newtheorem{theorem}{Theorem}[section]
\newtheorem{lemma}[theorem]{Lemma}
\newtheorem{proposition}[theorem]{Proposition}
\newtheorem{corollary}[theorem]{Corollary}

\theoremstyle{definition}
\newtheorem{defn}[theorem]{Definition}
\newtheorem{example}[theorem]{Example}
\newtheorem{remark}[theorem]{Remark}

\numberwithin{equation}{section}

\allowdisplaybreaks

\begin{document}

\title[Gevrey Regularity and Compact Quantum Metric Spaces]{Gevrey Regularity and Compact Quantum Metric Spaces for $L^p$-Group Algebras}

\author{Lin Chen$^*$}
\address{Department of Mathematics and Statistics, Suzhou University of Technology,  Suzhou, $215500$, P. R. China}
\email{linchen198112@163.com}

\author{Qin Wang}
\address{Department of Mathematics,  East China Normal University,
Shanghai, $200241$,  P. R. China}
\email{qwang@math.ecnu.edu.cn}

\thanks{$^*$Corresponding author: Lin Chen}



\begin{abstract}
We introduce the 
$\beta$-Gevrey $\ell^p$-rapid decay
property $(GRD)_{\beta,p}$, for $0<\beta\le 1$ and $1\le p<\infty$,
for countable discrete groups. This property is a subexponential analogue of classical rapid decay, in which polynomial control is replaced by logarithmic subexponential control of order 
$o(R^\beta)$. We establish basic results for $(GRD)_{\beta,p}$.
We then apply this framework to compact quantum metric structures on reduced 
$L^p$-group algebras. We introduce strongly dense-core 
$\beta$-Gevrey regular $\ell^p$-spectral triples and give two class of examples. For countable discrete groups satisfying $(GRD)_{\beta,p}$, we prove, using Rieffel’s criterion, that the corresponding Gevrey seminorms induce metrics on the Banach-algebra state space which metrize the weak-$*$ topology. This yields compact quantum metric space structures in settings beyond classical rapid decay, including groups of intermediate growth such as the first Grigorchuk group.
\end{abstract}

\maketitle

\section{Introduction and preliminaries}

The theory of compact quantum metric spaces, initiated by Connes'
spectral approach to noncommutative geometry and developed systematically by Rieffel \cite{Conn, Rief1, Rief2, Rief3, Rief4}, provides a framework for studying noncommutative spaces from a metric point of view. In this
framework, a seminorm on a unital $C^*$-algebra plays the role of the
classical Lipschitz seminorm on the algebra of Lipschitz functions on a
compact metric space. When the corresponding Monge--Kantorovich metric on
the state space induces the weak-$*$ topology, the seminorm is called a
\emph{Lip-norm}, and the resulting pair is regarded as a \emph{compact quantum metric
space}. This theory has been further developed by Li \cite{li} and by
Latrémolière through quantum Gromov-Hausdorff-type frameworks
\cite{Latr, Latr1}.

Group $C^*$-algebras form an important class of examples in this
theory. Ozawa and Rieffel \cite{Ozawa} proved that, for groups satisfying
a Haagerup-type condition, including word-hyperbolic groups, the reduced
group $C^*$-algebras admit compact quantum metric space structures.
Another fundamental analytic tool in this setting is the \emph{rapid decay
property (RD)}, which relates the algebraic length of group elements to the
operator-norm behavior of convolution operators on $\ell^2(G)$. This
property was first established for the free group $\mathbb F_2$ by
Haagerup \cite{haa}, and its general theory was developed by Jolissaint
\cite{Jo1,Jo2}. For surveys, see Chatterji \cite{chat4} and Garncarek
\cite{luka2}.

Rapid decay property are particularly useful for controlling seminorms
associated with length functions and for verifying, under suitable
hypotheses, that such seminorms define Lip-norms on reduced group
$C^*$-algebras. Antonescu and Christensen \cite{Anto} constructed a
large class of Lip-norms for groups with the rapid decay property. Long and Wu \cite{long} treated analogous constructions in
the twisted setting for discrete groups with rapid decay property.
Related quantum groups with rapid decay property were studied by Bhowmick, Voigt and Zacharias \cite{Bhow}. For
finitely generated amenable groups,  rapid decay property is equivalent to
polynomial growth. Christ and Rieffel \cite{ChriRi} showed that, under a
strong polynomial growth condition on the length function, the associated
metric induces the weak-$*$ topology on the state space. Austad,
Kaad and Kyed \cite{Austad} recently studied quantum metric structures
for crossed products by groups of polynomial growth.

Compact quantum metric spaces are also closely related to Connes'
noncommutative  geometry through \emph{spectral triples}. Given a
spectral triple \ensuremath{(\mathcal A,\mathcal H,D)}, the formula
$L_D(a)=\|[D,a]\|$
defines a natural seminorm on the elements \ensuremath{a\in\mathcal A}
for which the commutator \ensuremath{[D,a]} is bounded. Such seminorms
often, but not automatically, give compact quantum metric space
structures. In the group case, spectral triples associated with length
functions were introduced by Connes \cite{Conn} and remain a central source of
examples connecting group geometry, operator algebras, and noncommutative
metric geometry.

This naturally leads to the question of whether compact quantum metric
space structures arising from spectral triples admit analogues beyond the
Hilbert space and $C^*$-algebraic setting, in particular for
\emph{$L^p$-operator algebras}. For $p\in[1,\infty)$, an
$L^p$-operator algebra is a Banach algebra admitting an isometric
representation on some $L^p$-space. Such algebras were first considered by
Herz \cite{Herz}, who studied the $L^p$-operator algebra generated by the
left regular representation of a locally compact group.
Their systematic study was revived by Phillips \cite{Phi1, Phi2, Phi3} in the 2010s as part of a
program aimed at extending aspects of the modern theory of
$C^*$-algebras to the $L^p$ setting. Since then, $L^p$-operator algebras
have attracted considerable attention; see, for example,
\cite{Choi, Gard1, Gardella, Gard3}.

Classically, Gevrey regularity refines smoothness by imposing
quantitative bounds on derivatives of all orders. More precisely, let
$\Omega\subseteq\mathbb R^d$ be an open set and let $\sigma\geq 1$.
A function $u\in C^\infty(\Omega)$ is said to belong to the \emph{Gevrey class
of order $\sigma$} if, for every compact set $K\subset\Omega$, there
exist constants $C,R>0$ such that
$$\sup_{x\in K}|\partial^\alpha u(x)|
\leq C R^{|\alpha|}(\alpha!)^\sigma,
\qquad \alpha\in\mathbb N^d.$$
These classes were introduced by Maurice Gevrey in his 1918 work on
partial differential equations~\cite{Gev}.
This perspective admits an abstract Banach-algebraic analogue. For a
Banach algebra equipped with a closed derivation, ultradifferentiable and
Gevrey-type regularity can be expressed through growth estimates for the
iterates of the derivation. This point of view is closely related to the
Dales--Davie Banach algebra constructions of~\cite{DalesDavie1973} and
to the theory of inverse-closed smooth Banach subalgebras; see
\cite{Gr,Grk1,Gr2,Gr3,Grk5,Klo} and the references therein.

The Grigorchuk group was introduced by Grigorchuk as the first example of a finitely generated group of intermediate growth~\cite{Grig}. It is
amenable, but it does not have the rapid decay property with respect to a word-length function. Hence compact quantum metric space constructions based on rapid decay do not directly apply to such intermediate-growth
groups. Recently, Liao and Yu introduced property $(RD)_p$ and used it to prove that the $K$-theory of reduced $L^p$-operator algebras is
independent of the exponent $p\in[1,\infty)$~\cite{Liao}. Relatedly, the subexponential decay property (SD) for countable discrete groups has been studied in~\cite{Sri}. Motivated by these developments, we introduce the
\emph{$\beta$-Gevrey $\ell^p$-rapid decay property}, for
$0<\beta\leq 1$, denoted by $(GRD)_{\beta,p}$; see
Definition~\ref{le111}. Unlike the rapid decay property, which imposes
polynomial control, $(GRD)_{\beta,p}$ requires that, for functions
supported in balls of radius $R$, the operator norm of the left regular
representation on $\ell^p(G)$ be controlled by the $\ell^p$-norm up to a
factor whose logarithmic growth is $o(R^\beta)$. 

It is necessary to explain the choice of exponent $\beta$ in
our definition of $(GRD)_{\beta,p}$. For the Gevrey defining sequence
$M_k=(k!)^r$,
the associated function
$T_M(u)=\sup_{k\geq 0}\frac{u^k}{M_k}$
has growth equivalent to $\exp(cu^{1/r})$; see
Klotz~\cite[Section~3.2]{Klo}. Hence a factorial bound of order $(n!)^r$ corresponds to the
stretched-exponential scale $\exp(cu^{1/r})$, and the same scale appears
as $\exp(-cu^{1/r})$ when used as a decay estimate away from the
diagonal \cite{Gr2}. Taking $r=1/\beta$ gives the Gevrey bound $(n!)^{1/\beta}$ and the
decay scale $\exp(-cu^\beta)$. This is precisely the scale needed to
dominate volume growth satisfying
$\log F(R)=o(R^\beta)$.
Indeed, this condition implies that, for every $\varepsilon>0$,
we have $F(R)\leq \exp(\varepsilon R^\beta)$ for all sufficiently large
$R$. Thus $\exp(-cR^\beta)$ dominates the volume growth whenever
$\varepsilon<c$. In this sense, $\beta$ is the  decay exponent
dual to the Gevrey order $1/\beta$. The relationship between factorial bounds on derivatives and
subexponential decay is reflected in~\cite{Dasgup}, and also appears in our
work~\cite{chen}.

We first investigate the basic properties of groups satisfying
$(GRD)_{\beta,p}$. The main comparison results are summarized in
Table~\ref{tab:grd-implications}. We also establish stability under direct products (see Proposition \ref{stable}) and subgroups (see Proposition \ref{subgroup-stable}). 
 
\begin{table}[htbp]
	\centering
	\renewcommand{\arraystretch}{1.4}
	\begin{tabular}{|c|}
		\hline
		\textbf{Implication} \\
		\hline
		
		$(RD)_p \Longrightarrow (GRD)_{\beta,p}$ for all $0<\beta\leq 1$ (Proposition \ref{prop35})
		\\
		\hline
		
		$RD=(RD)_2 \Longrightarrow (GRD)_{\beta,2}$ for all $0<\beta\leq 1$
		\\
		\hline
		
		$\log |B_R|=o(R^\alpha) \ \text{for some} \ 0<\alpha\le 1
		\Longrightarrow
		(GRD)_{\beta,p}$
		for all $\alpha\leq\beta\leq 1$ \\ and $1\leq p<\infty$ (Remark  \ref{re34})
		\\
		\hline
		
		$SD\Longrightarrow (GRD)_{1,2}$ (Remark  \ref{re34})
		\\
		\hline
		
		amenable +$(GRD)_{\beta, p}$  for some $p\in(1,\infty)$ and $0<\beta\leq 1 \Longrightarrow
		\log |B_R|=o(R^\beta)$ \\(Proposition \ref{prop37})
		\\
		\hline
	\end{tabular}
	\caption{Implication relations for $(GRD)_{\beta,p}$.}
	\label{tab:grd-implications}
\end{table}

Recently, Delfin, Farsi, and Packer \cite{Del} constructed
\emph{$L^p$-spectral triples} for reduced $L^p$-group algebras of countable
discrete groups with proper length functions. They further showed that,
under a bounded doubling condition on the length function, these
$L^p$-spectral triples are metric. Building on this work, we adopt a slightly modified definition of an
$L^p$-spectral triple; see Definition~\ref{def43}. This formulation differs slightly from 
\cite[Definition 3.2]{Del} and fits
naturally with the standard theory of closed operators with compact
resolvent on Banach spaces. It is based on the common structural core
underlying the usual definitions of spectral triples on Hilbert spaces. Moreover, we introduce \emph{strongly dense-core
$\beta$-Gevrey regular $L^p$-spectral triples}; see
Definition~\ref{def35}. The condition in this definition requires the existence of a norm-dense subalgebra on which the derivation determined by the spectral triple can be iterated to all orders, with the norms of the resulting iterates satisfying Gevrey-type factorial estimates for every positive choice of the parameter. Our second main results provide two classes of
examples satisfying this definition.

\bigskip 

{\it{ \noindent \bf Example A} (cf. Theorem~\ref{th1}).
Let $1\leq p<\infty$. Consider the algebra $\mathcal A=C(\mathbb T)$ and the Banach space $L^p(\mu)=L^p(\mathbb T)$, where
$\mathbb T=\mathbb R/2\pi\mathbb Z$ is equipped with the normalized
Lebesgue measure $\mu=dx/(2\pi)$. Let $C(\mathbb T)$ act on $L^p(\mathbb T)$ via the multiplication operator
$M_a\xi=a\xi$, for $a\in C(\mathbb T)$ and $\xi\in L^p(\mathbb T)$.
Let $D = -i \frac{d}{dx}$  be the differential operator. Then, the domain of $D$ is 
\begin{align*}
\operatorname{Dom}(D) 
&=W^{1,p}(\mathbb T)\\
&=\{f\in L^p(0,2\pi):f(0)=f(2\pi), f\text{is absolutely continuous}, f^\prime\in L^p(0,2\pi)\}
\end{align*}
Moreover, for every $\beta>0$, $(C(\mathbb{T}), L^p(\mathbb{T}), D)$ is a strongly  dense-core $\beta$-Gevrey regular $L^p$-spectral triple.}

\medskip

{\it{ \noindent \bf Example B} (cf. Theorem~\ref{pro5}). 
Let $1\le p<\infty$, and let $G$ be a countable discrete group equipped
with a proper length function $\ell:G\to[0,\infty)$. Let	$D_\ell:\operatorname{Dom}(D_\ell)\subseteq \ell^p(G)\to \ell^p(G)$ be
the multiplication operator $$(D_\ell\xi)(x)=\ell(x)\xi(x),$$	
with domain
$$\operatorname{Dom}(D_\ell)=\{\xi\in \ell^p(G):\ell\xi\in \ell^p(G)\}.$$
Then, for every $\beta>0$, the triple
$(F_r^p(G),\ell^p(G),D_\ell)$ is a strongly dense-core
$\beta$-Gevrey regular $L^p$-spectral triple, with derivation
$\delta_\ell(\cdot)=[D_\ell,\cdot]$.}

\bigskip

The property $(GRD)_{\beta,p}$ extends the reach of the classical rapid
decay framework while retaining enough analytic control to yield compact
quantum metric space structures. Our third main result shows that, for
any  countable discrete group $G$ satisfying $(GRD)_{\beta,p}$, the strongly dense-core
$\beta$-Gevrey regular $L^p$-spectral triple
 $(F_r^p(G),\ell^p(G),D_\ell)$ of Example B induces such a structure.

\medskip

{\it{ \noindent \bf Theorem C} (cf. Theorem~\ref{main}).
For $p\in[1,\infty)$, let $G$ be a countable discrete group
equipped with a proper length function $\ell$, and let
$(F_r^p(G),\ell^p(G),D_\ell)$ be the associated strongly dense-core
$\beta$-Gevrey regular $L^p$-spectral triple. Assume that $G$ satisfies
$(GRD)_{\beta,p}$ for some $0<\beta\le 1$. Then, for every $C>0$, the
image of
	$$\mathcal B_{\beta,C}=\{a\in\lambda_p(\mathbb C[G]):L_{\beta,C}(a)\leq 1\}$$
in $F_r^p(G)/\mathbb C1$ is totally bounded with respect to the quotient
norm. Consequently, the metric
$$d_{\beta,C}(\omega,\psi)=\sup\{|\omega(a)-\psi(a)|:a\in\lambda_p(\mathbb C[G]),L_{\beta,C}(a)\leq 1\}$$	
metrizes the weak-$*$ topology on $S(F_r^p(G))$.}

\medskip

Combining Theorem C with Definition~\ref{de43}, we obtain, whenever
$G$ satisfies $(GRD)_{\beta,p}$ with respect to $\ell$, the associated
$L^p$-spectral triple $(F_r^p(G),\ell^p(G),D_\ell)$ is strongly $\beta$-metric; see
Corollary~\ref{maincro}. Our proof follows Rieffel's approach to compact quantum metric spaces. In contrast to Connes' spectral-triple perspective, Rieffel's construction
begins with metrics on state-like spaces associated with normed algebras.
Moreover, Rieffel's original formulation in \cite{Rief1} was given in
the general setting of normed spaces and did not rely on a $C^*$-algebra
structure. This flexibility makes Rieffel's framework particularly well
suited for studying compact quantum metric space structures in the
$L^p$-operator algebraic setting. The main technical point is to verify
the hypotheses of Rieffel's criterion \cite[Theorem~1.8]{Rief1}.

\bigskip

The paper is organized as follows. Section 2 recalls the basic definitions
 needed throughout the paper. Section 3 establishes
the basic properties of groups satisfying $(GRD)_{\beta,p}$. 
Section 4, we introduce strongly dense-core $\beta$-Gevrey regular
$L^p$-spectral triples and prove Example A and Example B. Section 5 proves
Theorem C by applying Rieffel's theorem, thereby showing that, when the
underlying group satisfies $(GRD)_{\beta,p}$, the associated strongly
dense-core $\beta$-Gevrey regular $L^p$-spectral triple is strongly
$\beta$-metric. Finally, Section 6 presents several examples.

\section{Preliminaries}
In this section we review some basic definitions.
\begin{defn}
Let $p\in[1,\infty)$, let $(X,\mathcal M,\mu)$ be a measure space, and let
$A$ be a Banach algebra. A \emph{representation} of $A$ on $L^p(\mu)$ is a
continuous algebra homomorphism
	$$\varphi:A\to\mathcal B(L^p(\mu)).$$
We say that $A$ is \emph{representable on $L^p(\mu)$} if there exists an
isometric representation
	$$\varphi:A\to\mathcal B(L^p(\mu)).$$
Such a representation is called \emph{$\sigma$-finite} if
$(X,\mathcal M,\mu)$ is $\sigma$-finite. We say that $A$ is
\emph{$\sigma$-finitely representable} if it admits a $\sigma$-finite
isometric representation.
\end{defn}
\begin{defn}
Let $p\in [1,\infty)$. A Banach algebra $A$ is an \emph{$L^p$-operator algebra}
if there is a measure space $(X,\mathcal M,\mu)$ such that $A$ is
representable on $L^p(\mu)$.
\end{defn}
Let $G$ be a countable discrete group and let $p\in[1,\infty)$.
For $f=(f(x))_{x\in G}\in \ell^1(G)$ and
$\xi=(\xi(x))_{x\in G}\in \ell^p(G)$, define $f*\xi:G\to\mathbb C$ by
$$(f*\xi)(x)=\sum_{y\in G} f(y)\xi(y^{-1}x),\qquad x\in G.$$
It is standard to check that $f*\xi\in \ell^p(G)$ and
$\|f*\xi\|_p\le \|f\|_1\|\xi\|_p$.
Thus we obtain the left regular representation
$$\lambda_p:\ell^1(G)\to B(\ell^p(G)),\qquad\lambda_p(f)\xi=f*\xi.$$
The \emph{reduced $L^p$ operator group algebra} of $G$ is defined by
$$F_r^p(G)=\overline{\lambda_p(\ell^1(G))}^{\,B(\ell^p(G))}.$$
Then $F_r^p(G)$ is a unital Banach subalgebra of $B(\ell^p(G))$,
with unit $\lambda_p(\delta_e)$, where  $e$ the identity element of $G$.
In particular, $F_r^p(G)$ is an $L^p$-operator algebra in its own right.
Moreover,  $\lambda_p(C_c(G))$  is a dense subalgebra of $F_r^p(G)$.
Let us recall the notion of a length function on a given group $G$..
\begin{defn}
A function $\ell:G\to \mathbb{R}_{+}$ is called a
\emph{length function} if it satisfies the following conditions:
	\begin{enumerate}
		\item[(i)] $\ell(x)=0$ if and only if $x=e$;
		\item[(ii)] $\ell(xy)\leq \ell(x)+\ell(y)$ for all $x,y\in G$;
		\item[(iii)] $\ell(x^{-1})=\ell(x)$ for all $x\in G$.
	\end{enumerate}
The length function $\ell$ is called \emph{proper} if
$\ell^{-1}([0,c])$ is finite for every $c\geq 0$.
\end{defn}

In analogy with the classical definition of rapid decay, we introduce the following definition.
\begin{defn}\label{le111}
Let $G$ be a countable discrete group equipped with a proper length function
$\ell:G\to[0,\infty)$. For $R\geq 0$, set $$B_R=\{x\in G:\ell(x)\leq R\}.$$
Let $p\in[1,\infty)$ and $0<\beta\leq 1$. We say that $G$ has the
\emph{$\beta$-Gevrey $\ell^p$-rapid decay property} with respect to $\ell$,
denoted by $(GRD)_{\beta,p}$, if there exists a function
$F:\mathbb R_{\geq 0}\to[1,\infty)$ such that
$$\log F(R)=o(R^\beta)\qquad\text{as} \ R\to\infty,$$
and such that, for every $R\geq 0$ and every $f\in\mathbb C[G]$ with
$\operatorname{supp}f\subseteq B_R$, we have
$$\|\lambda_p(f)\|_{B(\ell^p(G))}\leq F(R)\|f\|_{\ell^p(G)}.$$
Here $\lambda_p(f)$ denotes the left  regular representation on $\ell^p(G)$.
\end{defn}

\section{$\beta$-Gevrey $p$-rapid decay property}
In this section, we collect some basic results on countable discrete
groups with the $\beta$-Gevrey $p$-rapid decay property. Throughout the
section, for a finite set $E$, we denote its cardinality by $|E|$.
We first show that, 
for finitely generated groups, $(GRD)_{\beta,p}$ with respect to a word length
does not depend on the chosen finite generating set.

\begin{lemma}\label{lem-GRD-independent-word-length}
Let $\ell_1$ and $\ell_2$ be proper length functions on a countable discrete
group $G$ such that, for some $A\geq 1$ and $B\geq 0$,
	$$\ell_i(g)\leq A\ell_j(g)+B$$
for all $g\in G$ and $\{i,j\}=\{1,2\}$. Then $(GRD)_{\beta,p}$ with respect to
$\ell_1$ is equivalent to $(GRD)_{\beta,p}$ with respect to $\ell_2$, for all
$p\in[1,\infty)$ and $0<\beta\leq 1$.
\end{lemma}

\begin{proof}
If $(GRD)_{\beta,p}$ holds for $\ell_2$ with control function $F_2$, then
$$B_R^{\ell_1}\subseteq B_{AR+B}^{\ell_2}.$$
Thus $F_1(R)=F_2(AR+B)$ is a control function for $\ell_1$, since
	$$\log F_1(R)=\log F_2(AR+B)=o((AR+B)^\beta)=o(R^\beta).$$
The converse is identical.
\end{proof}

\begin{lemma}\label{monotone}
Let $G$ be a countable discrete group equipped with a proper length function $\ell$.
Let $p\in[1,\infty)$ and $0<\beta_1\leq\beta_2\leq 1$. If $G$ has
$(GRD)_{\beta_1,p}$ with respect to $\ell$, then $G$ has
$(GRD)_{\beta_2,p}$ with respect to $\ell$.
\end{lemma}

\begin{proof}
Let $F$ be a control function for $(GRD)_{\beta_1,p}$. Then for every
$R\geq 0$ and every $f\in\mathbb C[G]$ with
$\operatorname{supp}(f)\subseteq B_R$, we have
$$\|\lambda_p(f)\|_{\mathcal B(\ell^p(G))}\leq F(R)\|f\|_{\ell^p(G)}.$$
Moreover,
$\log F(R)=o(R^{\beta_1})$ as $R\to\infty$.
Since $\beta_1\leq\beta_2$, we have
	$$\frac{\log F(R)}{R^{\beta_2}}=\frac{\log F(R)}{R^{\beta_1}}\cdot\frac{1}{R^{\beta_2-\beta_1}}
	\longrightarrow 0.$$
Thus $\log F(R)=o(R^{\beta_2})$.
Hence the same function $F$ is also a control function for
$(GRD)_{\beta_2,p}$.
\end{proof}

\begin{proposition}\label{pro21}
Let $G$ be a countable discrete group equipped with a proper length function
$\ell$. Assume that, for some $0<\beta\leq 1$,
$\log |B_R|=o(R^\beta)$ as $R\to\infty$.
Then $G$ has $(GRD)_{\beta,p}$ with respect to $\ell$ for every
$p\in[1,\infty)$. 
\end{proposition}
\begin{proof}
Let $f \in \mathbb C[G]$  be finitely supported with $\text{supp}(f) \subseteq B_R$. By Hölder's inequality, we have
$$\|\lambda_p(f)\|_{B(\ell^p(G))} \leq \|f\|_{\ell^1(G)} = \sum_{g \in B_R} |f(g)| \le |B_R|^{1 - \frac{1}{p}} \|f\|_{\ell^p}. $$
Set $F_p(R)=|B_R|^{1-\frac1p}$.
Then
$$\log F_p(R)=\left(1-\frac1p\right)\log |B_R|=o(R^\beta)$$
as $R\to\infty$. Therefore $G$ satisfies $(GRD)_{\beta,p}$ with respect to
$\ell$.
\end{proof}

\begin{remark}\label{re34}
Let $G$ be a finitely generated group equipped with a word length $\ell$.
If $G$ has subexponential growth, that is,
$\log |B_R|=o(R)$,
then, by taking $\beta=1$ in Proposition~\ref{pro21}, $G$ satisfies
$(GRD)_{1,p}$ with respect to $\ell$ for every $p\in[1,\infty)$. In \cite{Sri}, the authors introduced the property $SD$, that is,  there exists a subexponential function $F$ such that, for every
$f\in\mathbb C[G]$ with supported in the $B_R$, 
$$\|\lambda_2(f)\|_{B(\ell^2(G))}\le F(R)\|f\|_{\ell^2(G)}.$$
Since $F$ is subexponential,
$\log F(R)=o(R)$.
Thus $G$ satisfies $(GRD)_{1,2}$.
More generally, if $\log |B_R|=o(R^\alpha)$
for some $0<\alpha\leq 1$, then Proposition~\ref{pro21} implies that
$G$ satisfies $(GRD)_{\alpha,p}$ with respect to $\ell$ for every
$p\in[1,\infty)$. Hence, by Proposition~\ref{monotone}, $G$ satisfies
$(GRD)_{\beta,p}$ with respect to $\ell$ for every
$\alpha\leq \beta\leq 1$ and every $p\in[1,\infty)$.
In view of Lemma~\ref{lem-GRD-independent-word-length}, for finitely generated
groups this property is independent of the chosen finite generating set.
\end{remark}

In \cite{Liao}, the authors introduced the property $(RD)_p$. More precisely,
a countable discrete group $G$ equipped with a length function $\ell$ is said
to have property $(RD)_p$ if there exist constants $C>0$ and $s\geq 0$ such
that, for every $R\geq 0$ and every $f\in\mathbb C[G]$ with
$\operatorname{supp}(f)\subseteq B_R$, we have
$$\|\lambda_p(f)\|_{\mathcal B(\ell^p(G))}\leq C(1+R)^s\|f\|_{\ell^p(G)}.$$
The following proposition shows that $(RD)_p$ implies $(GRD)_{\beta,p}$ for
all $0<\beta\leq 1$.

\begin{proposition}\label{prop35}
Let $G$ be a countable discrete group equipped with a proper length function $\ell$.
If $G$ satisfies $(RD)_p$, then $G$ satisfies $(GRD)_{\beta,p}$ with respect
to $\ell$ for every $0<\beta\leq 1$.
\end{proposition}

\begin{proof}
Let $C>0$ and $s\geq 0$ be constants as above. For every
$f\in\mathbb C[G]$ with $\operatorname{supp}(f)\subseteq B_R$, we have	$$\|\lambda_p(f)\|_{\mathcal B(\ell^p(G))}\leq C(1+R)^s\|f\|_{\ell^p(G)}.$$
Set $F(R)=C(1+R)^s$.
Then, for every $0<\beta\leq 1$,
$$\log F(R)=\log C+s\log(1+R)=o(R^\beta).$$
Thus $F$ is a control function for $(GRD)_{\beta,p}$. Hence $G$ satisfies
$(GRD)_{\beta,p}$ with respect to $\ell$ for every $0<\beta\leq 1$.
\end{proof}

The following proposition is analogous to \cite[Proposition 4.3]{Liao}.
It suggests that the notion of $(GRD)_{\beta,p}$ is most meaningful in the
range $p\in(1,2]$.

\begin{proposition}\label{pro-growth-grd}
Let $G$ be a countable discrete group equipped with a proper length function
$\ell$. If $G$ has $(GRD)_{\beta,p}$ with respect to $\ell$ for some
$p\in(2,\infty)$ and $0<\beta\leq 1$, then
$\log |B_R|=o(R^\beta)$ as $R\to\infty$.
\end{proposition}

\begin{proof}
Fix $R\geq 0$ and take $f=\chi_{B_R}$. Since
$\operatorname{supp} f\subseteq B_R$, the property $(GRD)_{\beta,p}$ gives
\begin{equation}\label{cheq1}
\|\lambda_p(\chi_{B_R})\|_{\mathcal B(\ell^p(G))}
	\leq F(R)\|\chi_{B_R}\|_{\ell^p(G)}=F(R)|B_R|^{1/p}.
\end{equation}
Now take $\xi=\chi_{B_R^{-1}}$.
Since 
$$(\lambda_p(\chi_{B_R})\xi)(e)=\sum_{y\in G}\chi_{B_R}(y)\xi(y^{-1})
	=\sum_{y\in G}\chi_{B_R}(y)\chi_{B_R^{-1}}(y^{-1})
	=|B_R|,$$
we have
$$\|\lambda_p(\chi_{B_R})\xi\|_{\ell^p(G)}\geq|(\lambda_p(\chi_{B_R})\xi)(e)|=|B_R|.$$
Hence 
\begin{equation}\label{cheneq2}
	\|\lambda_p(\chi_{B_R})\|_{\mathcal B(\ell^p(G))}
	\geq
	\frac{\|\lambda_p(\chi_{B_R})\xi\|_{\ell^p(G)}}{\|\xi\|_{\ell^p(G)}}
	\geq
	\frac{|B_R|}{|B_R|^{1/p}}
	=|B_R|^{1-\frac1p}.
\end{equation}
Combining \eqref{cheq1} and \eqref{cheneq2}, we get
	$$|B_R|^{1-\frac2p}\leq F(R).$$
Since $p>2$, we have $1-\frac2p>0$, and hence
$|B_R|\leq F(R)^{\frac{p}{p-2}}.$
Taking logarithms gives
	$$\log |B_R|\leq\frac{p}{p-2}\log F(R).$$
Since $(GRD)_{\beta,p}$ implies that
$\log F(R)=o(R^\beta)$ as $R\to\infty$, it follows that
$\log |B_R|=o(R^\beta)$ as $R\to\infty$.
\end{proof}

Analogous to the \cite[Corollary 3.18]{Jo1} and \cite[Proposition 2.8]{Sri}, we have the following result for amenable
groups.

\begin{proposition} \label{prop37}
Let $G$ be a countable discrete amenable group equipped with a proper length function $\ell$. If $G$ has $(GRD)_{\beta, p}$ with respect to $\ell$ for some $p\in(1,\infty)$ 
and $0<\beta\leq 1$, then  $\log |B_R|=o(R^\beta)$ as $R\to\infty$. 
\end{proposition}

\begin{proof}
Fix $R\geq 0$ and take $f=\chi_{B_R}$. 
The assumption  $G$ has $(GRD)_{\beta, p}$ gives
\begin{equation} \label{cheneq3}
\|\lambda_p(\chi_{B_R})\|_{B(\ell^p(G))} \le F(R)\|\chi_{B_R}\|_{\ell^p(G)} = F(R)|B_R|^{1/p}. 
\end{equation}
Since $G$ is amenable, for the finite subset $B_R \subseteq G$ and every $\varepsilon > 0$, there exists a finite subset $F_\varepsilon \subseteq G$ such that
$$ \frac{|xF_\varepsilon \triangle F_\varepsilon|}{|F_\varepsilon|} \le \varepsilon, $$ for all $ x \in B_R$.
Hence
$ |xF_\varepsilon \cap F_\varepsilon| \ge (1 - \varepsilon)|F_\varepsilon| $
for all $ x \in B_R$. Let $q$ be the conjugate exponent of $p$. Set
$$\xi_\varepsilon=|F_\varepsilon|^{-1/p}\chi_{F_\varepsilon}\in \ell^p(G),
\qquad\eta_\varepsilon=|F_\varepsilon|^{-1/q}\chi_{F_\varepsilon}\in \ell^q(G).$$
It is clear that $\|\xi_\varepsilon\|_{\ell^p(G)} = 1$ and $\|\eta_\varepsilon\|_{\ell^{q}(G)} = 1$. 
Hence, using the duality pairing between $\ell^p(G)$ and $\ell^q(G)$, we obtain
\begin{align*}
\langle \lambda_p(\chi_{B_R})\xi_\varepsilon, \eta_\varepsilon \rangle &= \sum_{x \in B_R} \langle \lambda_p(x)\xi_\varepsilon, \eta_\varepsilon \rangle \\
		&= |F_\varepsilon|^{-\frac{1}{p} - \frac{1}{q}} \sum_{x \in B_R} |xF_\varepsilon \cap F_\varepsilon| \\
		&= |F_\varepsilon|^{-1} \sum_{x \in B_R} |xF_\varepsilon \cap F_\varepsilon| \\
		&\ge |F_\varepsilon|^{-1} \sum_{x \in B_R} (1 - \varepsilon)|F_\varepsilon| \\
		&= (1 - \varepsilon)|B_R|.
\end{align*}
Taking the supremum yields
$\|\lambda_p(f)\|_{B(\ell^p(G))} \ge (1 - \varepsilon)|B_R|$.
Letting $\varepsilon \to 0$, we obtain 
\begin{equation}\label{cheneq4}
	\|\lambda_p(\chi_{B_R})\|_{B(\ell^p(G))} \ge |B_R|.
\end{equation}
Combining inequality \eqref{cheneq3} and \eqref{cheneq4}  yields 
$|B_R|\le F(R)|B_R|^{1/p}$.
Since $1<p<\infty$, the constant coefficient $(1-\frac{1}{p})$ is strictly positive, we get $$ |B_R|^{1-1/p} \le F(R).$$ Taking logarithms, we obtain
$\log |B_R|\leq\frac{p}{p-1}\log F(R)$.
Since $(GRD)_{\beta,p}$ implies that
$\log F(R)=o(R^\beta)$ as $R\to\infty$, it follows that $\log |B_R|=o(R^\beta)$ as  $R\to\infty$.
\end{proof}

The following proposition is analogous to \cite[Theorem 4.4]{Liao}.
\begin{proposition}\label{pr1}
Let $G$ be a countable discrete group equipped with a proper length function
$\ell$. If $G$ has $(GRD)_{\beta,p}$ with respect to $\ell$ for some
$p\in(1,\infty)$ and $0<\beta\leq 1$, then it has $(GRD)_{\beta,r}$
with respect to $\ell$ for every $r\in(1,p)$.
\end{proposition}

\begin{proof}
Let $1<r<p$. By assumption, for every $f \in \mathbb{C}[G]$ with $\mathrm{supp}(f) \subseteq B_R$, there exists a  function $F:\mathbb R_{\geq 0}\to[1,\infty)$ such that
$$\|\lambda_p(f)\|_{\mathcal B(\ell^p(G))}
	\le F(R)\|f\|_{\ell^p(G)}$$
and $\log F(R)=o(R^\beta)$ as $R\to\infty$.
Fix $R\ge 0$. Since $\ell$ is proper, $B_R$ is finite. Define 
$$\mathbb C^{B_R}=\{f\in\mathbb \mathbb C[G]:\operatorname{supp}(f)\subseteq B_R\}.$$
Consider the bilinear map
$T_R(f,\xi)=f*\xi$.	First, by Young's inequality, for $f\in\mathbb C^{B_R}$ and $\xi\in \mathbb C[G]$,
$$\|T_R(f,\xi)\|_{\ell^1(G)}
=\|f*\xi\|_{\ell^1(G)}\le\|f\|_{\ell^1(G)}\|\xi\|_{\ell^1(G)}.$$
Hence
$T_R:\ell^1(B_R)\times\ell^1(G)\to\ell^1(G)$
has bilinear norm at most $1$. Second, by the $(GRD)_{\beta,p}$ assumption,
for every $f\in\mathbb C^{B_R}$ and  $\xi\in \mathbb C[G]$,
\begin{align*}
\|T_R(f,\xi)\|_{\ell^p(G)}
&=\|f*\xi\|_{\ell^p(G)}\\
&=\|\lambda_p(f)\xi\|_{\ell^p(G)}\\
&\le\|\lambda_p(f)\|_{\mathcal B(\ell^p(G))}\|\xi\|_{\ell^p(G)}\\
&\le F(R)\|f\|_{\ell^p(G)}\|\xi\|_{\ell^p(G)}.
\end{align*}
By density of $\mathbb C[G]$ in $\ell^p(G)$, this estimate extends to all
$\xi\in\ell^p(G)$. Therefore $T_R:\ell^p(S)\times\ell^p(G)\to\ell^p(G)$
has bilinear norm at most $F(R)$.
Choose $\theta\in(0,1)$ such that
$$\frac1r=\frac{1-\theta}{1}+\frac{\theta}{p}.$$
Equivalently, $\theta=\frac{1-\frac1r}{1-\frac1p}$.
Since $1<r<p$, indeed $0<\theta<1$.
By the bilinear Riesz-Thorin interpolation theorem \cite{Calde},
$T_R:\ell^r(S)\times\ell^r(G)\to\ell^r(G)$
has bilinear norm at most
$$1^{1-\theta}F(R)^\theta=F(R)^\theta.$$
Therefore, for every $f\in\mathbb C^{B_R}$ and every
$\xi\in\ell^r(G)$,
$$\|f*\xi\|_{\ell^r(G)}\le F(R)^\theta \|f\|_{\ell^r(G)}\|\xi\|_{\ell^r(G)}.$$
Now let $f\in\mathbb \mathbb C[G]$ satisfy
$\operatorname{supp}(f)\subseteq B_R$.
Then $f\in\mathbb C^{B_R}$, and hence
$$\|f*\xi\|_{\ell^r(G)}\le	F(R)^\theta	\|f\|_{\ell^r(G)}	\|\xi\|_{\ell^r(G)}	$$
for every $\xi\in\ell^r(G)$. Therefore
$$\|\lambda_r(f)\|_{\mathcal B(\ell^r(G))}\le F(R)^\theta\|f\|_{\ell^r(G)}.$$
Set $F_r(R)=F(R)^\theta$.
Since $F(R)\geq 1$, we have $F_r(R)\geq 1$. Moreover,
$$\log F_r(R)=\theta\log F(R)=o(R^\beta)\qquad\text{as} \ R\to\infty.$$
Thus $G$ has $(GRD)_{\beta,r}$ with respect to $\ell$.
\end{proof}

The following result shows $(GRD)_{\beta,p}$ is  stable under direct products.
\begin{proposition}\label{stable}
Let $G$ and $H$ be countable discrete groups equipped with proper length functions
$\ell_G$ and $\ell_H$, respectively. Let $p\in[1,\infty)$ and	$0<\beta\leq 1$. If $G$ and $H$ have $(GRD)_{\beta,p}$ with respect to
$\ell_G$ and $\ell_H$, respectively, with control functions $F_G$ and $F_H$,
then $G\times H$, endowed with the length function
	$$\ell_{G\times H}(g,h)=\ell_G(g)+\ell_H(h),$$
has $(GRD)_{\beta,p}$ with control function $F_{G\times H}(R)=F_G(R)F_H(R)$.
\end{proposition}

\begin{proof}
For any $R>0$, we denote the closed balls in $G$ and $H$ by
$$B_R^G=\{g \in G:\ell_G(g)\le R\}, \qquad B_R^H = \{h \in H:\ell_H(h)\le R\}.$$
Let $\phi\in\mathbb{C}[G\times H]$ with $\mathrm{supp}(\phi)\subseteq B_R^{G\times H}$, where
$$B_R^{G\times H}=\{(g, h)\in G\times H : \ell_G(g)+\ell_H(h)\le R\}.$$ 
For each $g\in G$, we define $\phi_g(h)=\phi(g, h)$. If $\phi_g\neq 0$, then 
$\ell_G(g)\le R$ and $\mathrm{supp}(\phi_g)\subseteq B_{R-\ell_G(g)}^H\subseteq B_R^H$.
For each $g \in G$, define
	$$A(g)=\|\lambda_p^H(\phi_g)\|_{B(\ell^p(H))}.$$
Since $H$ has  $(GRD)_{\beta, p}$, we have
$ A(g) \le F_H(R)\|\phi_g\|_{\ell^p(H)}$.
Since $\phi$ has finite support, $A$ is finitely supported and hence
may be regarded as an element of $\mathbb C[G]$. Moreover, it is clear that $\mathrm{supp}(A) \subseteq B_R^G$, we obtain
$$\|A\|_{\ell^p(G)}\le F_H(R)\left(\sum_{g\in G}\|\phi_g\|_{\ell^p(H)}^p\right)^{1/p}.$$
By Fubini's theorem for discrete sums, the sum on the right-hand side is exactly $\|\phi\|_{\ell^p(G \times H)}$. Hence,
\begin{equation}\label{cheneq7}
\|A\|_{\ell^p(G)} \le F_H(R)\|\phi\|_{\ell^p(G \times H)}. 
\end{equation}
Let $\xi\in\ell^p(G\times H)$. For each $x\in G$, define $\xi_x(k)=\xi(x, k)$ for $k\in H$, and set $\eta(x)=\|\xi_x\|_{\ell^p(H)}$. 
Then $\|\eta\|_{\ell^p(G)}=\|\xi\|_{\ell^p(G\times H)}$. 
Now, for fixed $x\in G$ and $k\in H$, 
\begin{align*}
(\lambda_p^{G \times H}(\phi )\xi)(x, k) 
&=\sum_{g\in G}\sum_{h\in H} \phi(g, h)\xi(g^{-1}x, h^{-1}k) \\
&=\sum_{g\in G}(\lambda_p^H(\phi_g)\xi_{g^{-1}x})(k).
\end{align*} 
Taking the $\ell^p(H)$-norm with respect to the variable $k$, we get
\begin{align*}
\|(\lambda_p^{G\times H}(\phi)\xi)_x\|_{\ell^p(H)} 
&\le\sum_{g\in G}\|\lambda_p^H(\phi_g)\xi_{g^{-1}x}\|_{\ell^p(H)} \\
&\le\sum_{g\in G} A(g)\|\xi_{g^{-1}x}\|_{\ell^p(H)} \\
&=\sum_{g\in G} A(g)\eta(g^{-1}x)=(A * \eta)(x).
\end{align*} 
Taking the $\ell^p(G)$-norm with respect to $x$ yields
\begin{equation}\label{cheneq8}
\|\lambda_p^{G \times H}(\phi)\xi\|_{\ell^p(G \times H)} \le \|A * \eta\|_{\ell^p(G)}. 
\end{equation}
Since $G$ has $(GRD)_{\beta, p}$ and $\mathrm{supp}(A) \subseteq B_R^G$, we have
\begin{equation} \label{cheneq9}
\|A * \eta\|_{\ell^p(G)} \le F_G(R)\|A\|_{\ell^p(G)}\|\eta\|_{\ell^p(G)}. 
\end{equation}
By \eqref{cheneq7}, \eqref{cheneq8} and \eqref{cheneq9}, we obtain
$$\|\lambda_p^{G\times H}(\phi)\xi\|_{\ell^p(G \times H)}\le F_G(R) \Big( F_H(R)\|\phi\|_{\ell^p(G\times H)} \Big)\|\xi\|_{\ell^p(G\times H)}. $$
Taking the supremum over all nonzero $\xi \in \ell^p(G\times H)$ gives
$$\|\lambda_p^{G\times H}(\phi)\|_{B(\ell^p(G\times H))} \le F_G(R)F_H(R)\|\phi\|_{\ell^p(G \times H)}. $$
Since $$\log\left(F_G(R)F_H(R)\right)=\log\left(F_G(R)\right)+\log\left(F_H(R)\right)=o(R^\beta)$$so $G\times H$ satisfies $(GRD)_{\beta, p}$ with the control function 
$F_{G \times H}(R)=F_G(R)F_H(R)$.
\end{proof}

Proposition \ref{stable} immediately yields the following corollaries.
\begin{corollary}
Let $G$ and $H$ be countable discrete groups equipped with proper length
functions $\ell_G$ and $\ell_H$, respectively. Let $p\in[1,\infty)$.
Suppose that $G$ has $(GRD)_{\beta_G,p}$ with respect to $\ell_G$ and
$H$ has $(GRD)_{\beta_H,p}$ with respect to $\ell_H$, where
$0<\beta_G,\beta_H\leq 1$. Then $G\times H$, endowed with the length function
$$\ell_{G\times H}(g,h)=\ell_G(g)+\ell_H(h),$$
has $(GRD)_{\max\{\beta_G,\beta_H\},p}$.
\end{corollary}
\begin{proof}
Set $\beta=\max\{\beta_G,\beta_H\}$. Since $\beta_G\leq \beta$ and
$\beta_H\leq \beta$, by Lemma \ref{monotone}, $G$ and $H$ both have
$(GRD)_{\beta,p}$ with respect to $\ell_G$ and $\ell_H$, respectively.
By Proposition~\ref{stable}, the direct product $G\times H$, endowed with
	$$\ell_{G\times H}(g,h)=\ell_G(g)+\ell_H(h),$$
has $(GRD)_{\beta,p}$. That is, $G\times H$ has $(GRD)_{\max\{\beta_G,\beta_H\},p}$.
\end{proof}

\begin{corollary}
Let $G$ and $H$ be countable discrete groups equipped with proper length functions
$\ell_G$ and $\ell_H$, respectively.    
For $p, p^\prime\in[1,\infty)$ and $0<\beta\leq 1$. If $G$ have $(GRD)_{\beta, p}$ with respect to $\ell_G$ and $H$ have 
$(GRD)_{\beta, p^\prime}$ with respect to
$\ell_H$, then the direct product $G\times H$, endowed with the length
$$\ell_{G\times H}(g,h)=\ell_G(g)+\ell_H(h),$$
has $(GRD)_{\beta, r}$, where $r=\min{(p,p^\prime)}$.
\end{corollary}
\begin{proof}
Set $r=\min\{p,p^\prime\}$. If $r=1$, then the conclusion follows from Young's inequality. Indeed,
for every finitely supported function $\phi$ on $G\times H$ and every
$\xi\in\ell^1(G\times H)$, we have
$$\|\phi*\xi\|_{\ell^1(G\times H)}
	\le
\|\phi\|_{\ell^1(G\times H)}\|\xi\|_{\ell^1(G\times H)}.$$
Hence
$$\|\lambda_1^{G\times H}(\phi)\|_{\mathcal B(\ell^1(G\times H))}
	\le\|\phi\|_{\ell^1(G\times H)}.$$
Thus $G\times H$ has $(GRD)_{\beta,1}$ with control function
$F(R)=1$. Now assume that $r>1$. Since $r\le p$ and $G$ has $(GRD)_{\beta,p}$, Proposition \ref{pr1} implies that $G$ has $(GRD)_{\beta,r}$. Similarly, since $r\le p^\prime$ and $H$ has $(GRD)_{\beta,p^\prime}$, the group
$H$ has $(GRD)_{\beta,r}$.
By  Proposition \ref{stable}, the group
$G\times H$, equipped with
$\ell_{G\times H}(g,h)=\ell_G(g)+\ell_H(h)$, 
has $(GRD)_{\beta,r}$.
\end{proof}

The next Proposition  generalizes  Proposition 2.1.1 and 2.1.5 of \cite{Jo1}.
\begin{proposition}\label{subgroup-stable}
Let $G$ be a countable discrete group equipped with a proper length function
$\ell_G$, and let $p\in[1,\infty)$. Let $H\leq G$ be a subgroup, equipped
with the restricted length function $\ell_H=\ell_G|_H$. Then the following
hold.
	\begin{enumerate}
		\item[(i)] If $G$ has $(GRD)_{\beta,p}$ with respect to $\ell_G$ for some
		$0<\beta\leq 1$, then $H$ has $(GRD)_{\beta,p}$ with respect to $\ell_H$.
		
		\item[(ii)] If $0<\beta\leq 1$ and $[G:H]<\infty$, then $G$ has
		$(GRD)_{\beta,p}$ with respect to $\ell_G$ if and only if $H$ has
		$(GRD)_{\beta,p}$ with respect to $\ell_H$.
	\end{enumerate}
\end{proposition}

\begin{proof}
(i) Since $\ell_H=\ell_G|_H$, we have $B_R^H=B_R^G\cap H$.
Let $f\in\mathbb C[H]$ satisfy
$\operatorname{supp}(f)\subseteq B_R^H$. Viewing $f$ as an element of
$\mathbb C[G]$, we have $\operatorname{supp}(f)\subseteq B_R^G$.
We claim that
\begin{equation}\label{eq312}
	\|\lambda_p^G(f)\|_{\mathcal B(\ell^p(G))}=\|\lambda_p^H(f)\|_{\mathcal B(\ell^p(H))}.
\end{equation}
Indeed, decompose $G$ into cosets of the form $Hx$, that is 
	$G=\bigsqcup_{x\in H\backslash G}Hx$.
Then $$\ell^p(G)\cong \bigoplus_{x\in H\backslash G}^{p}\ell^p(Hx).$$
Since $f$ is supported in $H$, the operator $\lambda_p^G(f)$ preserves each
subspace $\ell^p(Hx)$. Under the isometric identification
$\ell^p(Hx)\cong \ell^p(H)$, its restriction to $\ell^p(Hx)$ is precisely
$\lambda_p^H(f)$. Hence $\lambda_p^G(f)$ is the $\ell^p$-direct sum of copies
of $\lambda_p^H(f)$, and the claimed equality follows.
If $F_G$ is a control function for $(GRD)_{\beta,p}$ on $G$, then
	$$\|\lambda_p^H(f)\|_{\mathcal B(\ell^p(H))}
	=\|\lambda_p^G(f)\|_{\mathcal B(\ell^p(G))}\leq
	F_G(R)\|f\|_{\ell^p(G)}=F_G(R)\|f\|_{\ell^p(H)}.$$
Therefore $H$ has $(GRD)_{\beta,p}$ with respect to $\ell_H$.
	
(ii) The implication from $G$ to $H$ follows from (i). Conversely, assume that
$[G:H]=k<\infty$ and that $H$ has $(GRD)_{\beta,p}$ with respect to
$\ell_H$. Choose representatives $s_1,\dots,s_k$ such that
	$$G=\bigsqcup_{i=1}^k Hs_i.$$
Set $M=\max\limits_{1\leq i\leq k}\ell_G(s_i^{-1})$.
Let $\phi\in\mathbb C[G]$ satisfy $\operatorname{supp}(\phi)\subseteq B_R^G$.
Write $\phi=\sum_{i=1}^k \phi_i$, where $\phi_i$ is supported in $Hs_i$. For each $i$, define
$\widetilde\phi_i\in\mathbb C[H]$ by
$\widetilde\phi_i(h)=\phi_i(hs_i)$.
Then $\phi_i=\widetilde\phi_i*\delta_{s_i}$.
If $h\in\operatorname{supp}(\widetilde\phi_i)$, then
$hs_i\in\operatorname{supp}(\phi)\subseteq B_R^G$.
Therefore
	$$\ell_H(h)=\ell_G(h)\leq\ell_G(hs_i)+\ell_G(s_i^{-1})\leq R+M.$$
Hence
$\operatorname{supp}(\widetilde\phi_i)\subseteq B_{R+M}^H$.
Since $\lambda_p^G(\delta_{s_i})$ is an isometry on $\ell^p(G)$, we obtain
\begin{align*}
\|\lambda_p^G(\phi)\|_{\mathcal B(\ell^p(G))}
&\leq\sum_{i=1}^k\|\lambda_p^G(\widetilde\phi_i*\delta_{s_i})\|_{\mathcal B(\ell^p(G))} \\
&\leq\sum_{i=1}^k\|\lambda_p^G(\widetilde\phi_i)\|_{\mathcal B(\ell^p(G))}.
\end{align*}
By the Eq. \eqref{eq312},  we have
$\|\lambda_p^G(\widetilde\phi_i)\|_{\mathcal B(\ell^p(G))}
	=\|\lambda_p^H(\widetilde\phi_i)\|_{\mathcal B(\ell^p(H))}$.
Thus, if $F_H$ is a control function for $H$, then
$$\|\lambda_p^G(\phi)\|_{\mathcal B(\ell^p(G))}
	\leq F_H(R+M)\sum_{i=1}^k\|\widetilde\phi_i\|_{\ell^p(H)}.$$
By Hölder's inequality,
$$\sum_{i=1}^k\|\widetilde\phi_i\|_{\ell^p(H)}
	\leq k^{1-\frac1p}\left(\sum_{i=1}^k\|\widetilde\phi_i\|_{\ell^p(H)}^p\right)^{1/p}.$$
Since the supports of the $\phi_i$ are disjoint and
$\|\widetilde\phi_i\|_{\ell^p(H)}=\|\phi_i\|_{\ell^p(G)}$,
we have
$\left(\sum_{i=1}^k\|\widetilde\phi_i\|_{\ell^p(H)}^p\right)^{1/p}=\|\phi\|_{\ell^p(G)}$.
Therefore
$$\|\lambda_p^G(\phi)\|_{\mathcal B(\ell^p(G))}
	\leq
k^{1-\frac1p}F_H(R+M)\|\phi\|_{\ell^p(G)}.$$
Define $F_G(R)=k^{1-\frac1p}F_H(R+M)$.
Since $k$ and $M$ are fixed and $\log F_H(R)=o(R^\beta)$,
we get
$$\log F_G(R)=\left(1-\frac1p\right)\log k+\log F_H(R+M)=o(R^\beta).$$
Hence $G$ has $(GRD)_{\beta,p}$ with respect to $\ell_G$.
\end{proof}

\section{ $\beta$-Gevrey regular $L^p$-spectral triple}

In this section, we introduce the notions of an $L^p$-spectral triple
and a strongly dense-core $\beta$-Gevrey regular $L^p$-spectral triple,
and present two basic examples.
\begin{defn}\label{cldef1}
Let $1\leq p<\infty$, and let
$$D:\operatorname{Dom}(D)\subseteq L^p(\mu)\to L^p(\mu)$$
be a densely defined closed operator. We define an operator
$\delta_D$ on $\mathcal B(L^p(\mu))$ as follows. Its domain is
\begin{align*}
\operatorname{Dom}(\delta_D)
&=\Big\{
	T\in \mathcal B(L^p(\mu)) :
	T(\operatorname{Dom}(D))\subseteq \operatorname{Dom}(D)
	\text{ and }\\
&[D,T]\big|_{\operatorname{Dom}(D)}
	\text{ extends to a bounded operator on } L^p(\mu)
	\Big\}.
\end{align*}
For $T\in\operatorname{Dom}(\delta_D)$, define $\delta_D(T)$ to be
the unique bounded extension of
	$[D,T]\big|_{\operatorname{Dom}(D)}$ to $L^p(\mu)$.
The iterated domains are defined recursively by
$\operatorname{Dom}(\delta_D^1)=\operatorname{Dom}(\delta_D)$,
and, for $n\geq 1$,
$$\operatorname{Dom}(\delta_D^{n+1})=\{T\in\operatorname{Dom}(\delta_D^n):
\delta_D^n(T)\in\operatorname{Dom}(\delta_D)\}.$$
On this domain, we set
$\delta_D^{n+1}(T)=\delta_D(\delta_D^n(T))$.
\end{defn}

\begin{lemma}\label{closed-derivation}
The operator $\delta_D$ defined in Definition~\ref{cldef1} is a closed
derivation on $\mathcal B(L^p(\mu))$. More precisely,
$\operatorname{Dom}(\delta_D)$ is a subalgebra of
$\mathcal B(L^p(\mu))$, and for all $A,B\in\operatorname{Dom}(\delta_D)$,
$$\delta_D(AB)=\delta_D(A)B+A\delta_D(B).$$
\end{lemma}

\begin{proof}
The Leibniz rule follows immediately from the identity
	$$[D,AB]=[D,A]B+A[D,B]$$
on $\operatorname{Dom}(D)$. This also shows that
$\operatorname{Dom}(\delta_D)$ is a subalgebra of
$\mathcal B(L^p(\mu))$.
It remains to prove that $\delta_D$ is closed. Suppose that
$T_n\in\operatorname{Dom}(\delta_D)$, $T_n\to T$ and
$\delta_D(T_n)\to S$ in operator norm. Then, for $x\in\operatorname{Dom}(D)$,
$$DT_nx=T_nDx+\delta_D(T_n)x.$$ 
Since $T_nx\to Tx$ and
$$T_nDx+\delta_D(T_n)x\to TDx+Sx,$$
the closedness of $D$ implies that $Tx\in\operatorname{Dom}(D)$ and
$DTx=TDx+Sx$. Hence, $[D,T]x=Sx$ and  $x\in\operatorname{Dom}(D)$.
Therefore $T\in\operatorname{Dom}(\delta_D)$ and $\delta_D(T)=S$. 
\end{proof}

The closed derivation $\delta_D$ introduced above encodes the bounded
commutator condition in the classical definition of spectral triples
\cite[Definition 1.1]{Latr}. Motivated by this classical framework,
we introduce the following notion of an $L^p$-spectral triple.

\begin{defn}\label{def43}
Let $p\in [1,\infty)$. An $L^p$-spectral triple
$(\mathcal A,L^p(\mu),D)$
consists of a unital $L^p$-operator algebra $\mathcal A$,
a faithful unital  representation 
$\varphi:\mathcal A\to\mathcal B(L^p(\mu))$,
and a densely defined closed operator
$D:\operatorname{Dom}(D)\subseteq L^p(\mu)\to L^p(\mu)$
satisfying the following conditions:
\begin{enumerate}
\item[(1)] The operator $D$ has compact resolvent, that is,
	$\rho(D)\neq\emptyset$ and there exists $\lambda_0\in\rho(D)$ such that
	$(D-\lambda_0 I)^{-1}\in\mathcal K(L^p(\mu))$,
		
\item[(2)] The set $\mathcal A_D=\left\{a\in\mathcal A:\varphi(a)\in\operatorname{Dom}(\delta_D)\right\}$
  is a norm-dense subalgebra of $\mathcal A$.
	\end{enumerate}
\end{defn}

\begin{remark}
Our definition gives a slightly different formulation of an
$L^p$-spectral triple from that in \cite[Definition 3.2]{Del}, and is adapted to the
standard theory of closed operators with compact resolvent on Banach
spaces. Since $D$ is unbounded, the commutator $[D,\varphi(a)]$ is not
automatically a bounded operator on $L^p(\mu)$. We therefore formulate
the bounded commutator condition in terms of the derivation $\delta_D$,
which makes explicit the required domain condition and the bounded
extension of the commutator. The compactness condition is stated in the standard compact-resolvent
form. Namely, we require $\rho(D)\neq\varnothing$ and
$(D-\lambda_0 I)^{-1}\in\mathcal K(L^p(\mu))$ for some
$\lambda_0\in\rho(D)$. By the resolvent identity, this is equivalent to
compactness of $(D-\lambda I)^{-1}$ for all $\lambda\in\rho(D)$.
\end{remark}

\medskip
Recall that a classical spectral triple $(\mathcal A,\mathcal H,D)$   on a
Hilbert space is called \emph{regular} if, for each $a\in\mathcal A$ and each
$k\in\mathbb N$, both $a$ and $[D,a]$ belong to
$\operatorname{Dom}(\delta^k)$, where
$$\delta(T)=[|D|,T].$$
Equivalently, the algebra generated by $\mathcal A$ and $[D,\mathcal A]$
is contained in the smooth domain $\bigcap_{n=1}^{\infty}\operatorname{Dom}(\delta^n)$;
see \cite[Definition 10.10]{Gracia}.
In the $L^p$-setting, the representation acts on the Banach space
$L^p(\mu)$. For $p\neq 2$, arbitrary closed unbounded operators generally do not have
a canonical absolute value $|D|$ analogous to the Hilbert space case. We therefore work directly with the closed derivation $\delta_D(T)=[D,T]$,
introduced in Definition \ref{cldef1}. This allows us to formulate Gevrey regularity in terms of estimates for the iterated commutators $\delta_D^n(\cdot)$.

\begin{defn}\label{def35}
Let $(\mathcal A,L^p(\mu),D)$ be an $L^p$-spectral triple with a faithful unital representation
$\varphi:\mathcal A\to\mathcal B(L^p(\mu))$ and the derivation $\delta_D$ defined in Definition \ref{cldef1}. Define the smooth domain of $\delta_D$ inside $\mathcal A$ by
$$\mathcal A^\infty=\left\{a\in\mathcal A:\varphi(a)\in
\bigcap_{n=1}^{\infty}\operatorname{Dom}(\delta_D^n)\right\}.$$
Fix \(\beta>0\) and \(C>0\).  
We define the \emph{Gevrey seminorm $L_{\beta, C}(a)$}  by
$$ L_{\beta, C}(a) = \sum_{n=1}^{\infty} \frac{\left\| \delta_D^n(\varphi(a)) \right\|_{\mathcal{B}(L^p(\mu))}}{C^n(n!)^{1/\beta}}. $$
We define the set of elements with finite Gevrey seminorm as
$$\mathcal{A}_{\beta, C} = \left\{a\in\mathcal A^\infty:L_{\beta, C}(a)	<\infty\right\}. $$
\end{defn}
\begin{remark}
The spaces $\mathcal A_{\beta,C}$ are increasing in the parameter $C$.
Indeed, if $0<C_1<C_2$ and $a\in\mathcal A_{\beta,C_1}$, then
$L_{\beta,C_2}(a)\leq L_{\beta,C_1}(a)<\infty$.
Thus $a\in\mathcal A_{\beta,C_2}$, and consequently
$\mathcal A_{\beta,C_1}\subseteq \mathcal A_{\beta,C_2}$.\
If $0<\beta\leq 1$, then $\mathcal A_{\beta,C}$ is a subalgebra of $\mathcal A$. Indeed, Let $a,b\in\mathcal A_{\beta,C}$. Since $a,b\in\mathcal A^\infty$ and
$\delta_D$ is a derivation, the iterated Leibniz formula gives
$$\delta_D^n(\varphi(ab))=\sum_{j=0}^{n}\binom{n}{j}
\delta_D^j(\varphi(a))\delta_D^{n-j}(\varphi(b)).$$
Hence
$$\|\delta_D^n(\varphi(ab))\|\leq\sum_{j=0}^{n}\binom{n}{j}\|\delta_D^j(\varphi(a))\|
\|\delta_D^{n-j}(\varphi(b))\|.$$
For $0<\beta\leq 1$, we have
$$\frac{\binom{n}{j}}{(n!)^{1/\beta}}
\leq\frac{1}{(j!)^{1/\beta}((n-j)!)^{1/\beta}}.$$
Therefore
$$\frac{\|\delta_D^n(\varphi(ab))\|}{C^n(n!)^{1/\beta}}
\leq\sum_{j=0}^{n}\frac{\|\delta_D^j(\varphi(a))\|}{C^j(j!)^{1/\beta}}
\frac{\|\delta_D^{n-j}(\varphi(b))\|}{C^{n-j}((n-j)!)^{1/\beta}}.$$
Summing over $n\geq 1$ gives
$$L_{\beta,C}(ab)\leq
\|\varphi(a)\|L_{\beta,C}(b)+\|\varphi(b)\|L_{\beta,C}(a)+L_{\beta,C}(a)L_{\beta,C}(b).$$
Since $a,b\in\mathcal A_{\beta,C}$, the right-hand side is finite. Thus
$ab\in\mathcal A_{\beta,C}$.
Linearity follows immediately from the linearity of $\delta_D^n$ and the
triangle inequality. Hence $\mathcal A_{\beta,C}$ is a subalgebra of
$\mathcal A$.
\end{remark}

\begin{defn}
We say that the $L^p$-spectral triple $(\mathcal A,L^p(\mu),D)$ is
\emph{strongly dense-core $\beta$-Gevrey regular} if there exists a norm-dense
subalgebra $\mathcal B\subseteq\mathcal A$ such that
$$\mathcal B\subseteq\mathcal A_{\beta,C}\quad\text{for every } C>0.$$
\end{defn}

\medskip 
The next example illustrates that our notion of an $L^p$-spectral triple
is compatible with classical examples arising in the theory of compact
Banach spectral triples. Recall that, in the sense of
\cite[Definition 5.10]{Arha1}, compact Banach spectral triples are built
from representations of algebras on Banach spaces, together with
unbounded closed bisectorial operators satisfying a suitable bounded
holomorphic functional calculus. The triple
$(C(\mathbb T), L^p(\mathbb T), -i\frac{d}{dx})$ is a basic example
in that framework; see \cite[section 8.1]{Arha}. The following proposition
shows that this example also satisfies our definition of an
$L^p$-spectral triple.

\begin{theorem}\label{th1}
Let $1\leq p<\infty$. Consider the algebra $\mathcal A=C(\mathbb T)$ and the Banach space $L^p(\mu)=L^p(\mathbb T)$, where $\mathbb T=\mathbb R/2\pi\mathbb Z$ is equipped with the normalized Lebesgue measure $\mu = \frac{dx}{2\pi}$. Let $C(\mathbb T)$ act on $L^p(\mathbb T)$ via the multiplication operator
$ M_a\xi=a\xi$, for $a\in C(\mathbb T) $ and $\xi\in L^p(\mathbb T)$.
Let $D = -i \frac{d}{dx}$  be the differential operator. Then, the domain of $D$ is 
\begin{align*}
\operatorname{Dom}(D) 
&=W^{1,p}(\mathbb T)\\
&=\{f\in L^p(0, 2\pi):f(0)=f(2\pi),~f~\text{is absolutely continuous}, f^\prime\in L^p(0, 2\pi)\}.
\end{align*}
Moreover, for every $\beta>0$, $(C(\mathbb{T}), L^p(\mathbb{T}), D)$ is a strongly  dense-core $\beta$-Gevrey regular $L^p$-spectral triple.
\end{theorem}

\begin{proof}

The differential operator $D=-i\frac{d}{dx}$  is closed with the $\operatorname{Dom}(D)=W^{1,p}(\mathbb T)$ can be found in \cite[page 29]{Arha}.
We first verify that $(C(\mathbb{T}), L^p(\mathbb{T}), D)$ forms an $L^p$-spectral triple.
Since the torus $\mathbb T$ is compact, every continuous function $a\in C(\mathbb T)$ is essentially bounded, meaning $a\in L^\infty(\mathbb T)$. 
By \cite[Proposition 4.10, page 31]{EnN00} any multiplication operator $M_a$ is bounded on $L^p(\mathbb T)$ with its operator norm 
$$\|M_a\|_{\mathcal B(L^p(\mathbb T))} = \|a\|_\infty.$$
Hence, the map $\varphi:C(\mathbb T)\to\mathcal B(L^p(\mathbb T))$ defined by 
$a\mapsto M_a$ is a faithful unital representation.
For any $\lambda\in\mathbb C$, the resolvent equation $(D-\lambda I)u=f$ is equivalent to
\begin{equation} \label{eq41}
	u'-i\lambda u=if.
\end{equation}
Solving this on the interval $[0,2\pi]$ yields
\begin{equation} \label{eq43}
u(x) = e^{i\lambda x} \left( c+i\int_0^x e^{-i\lambda t}f(t)\,dt \right),
\end{equation}
where $c\in\mathbb C$ is a constant. Imposing the periodic boundary condition $u(2\pi)=u(0)$ gives
\begin{equation} \label{eq42}
(1-e^{2\pi i\lambda})c = i e^{2\pi i\lambda} \int_0^{2\pi} e^{-i\lambda t}f(t)\,dt. 
\end{equation}
If $\lambda\notin\mathbb Z$, then $e^{2\pi i\lambda}\neq 1$, which uniquely determines the constant $c$ and thus provides a unique periodic solution $u\in W^{1,p}(\mathbb T)$. 
Furthermore, by H\"older's inequality, for any $x \in [0,2\pi]$ we 
\begin{align*}
\left| \int_0^x e^{-i\lambda t}f(t)dt \right| 
&\leq\int_0^{2\pi} \left| e^{-i\lambda t} \right| |f(t)|dt \\	
& \leq C_1\left( \int_0^{2\pi} |f(t)|^pdt \right)^{1/p} \left( \int_0^{2\pi} 1^qdt \right)^{1/q} = C_2 \|f\|_p. 
\end{align*}
Thus, by  \eqref{eq42}
$$|c|=\left| \frac{i e^{2\pi i\lambda}}{1-e^{2\pi i\lambda}}\right|\left|\int_0^{2\pi} e^{-i\lambda t}f(t)\,dt \right|\leq C_3\|f\|_p.$$
Consequently, by  \eqref{eq43}
$$|u(x)|\leq\left|e^{i\lambda x}\right|\left(|c|+\left|\int_0^x e^{-i\lambda t}f(t)dt \right|\right)\leq C_4\|f\|_p.$$
So,
$$\|u\|_p
=\left(\int_0^{2\pi}|u(x)|^p dx\right)^{1/p}\leq C_4\|f\|_p(2\pi)^{1/p}=C_\lambda \|f\|_p.$$
Combining this equality with  \eqref{eq41} yields
$$\|u'\|_p\leq |\lambda|\|u\|_p+\|f\|_p\leq C_\lambda'\|f\|_p.$$ 
Hence  $\|u\|_{W^{1,p}} \leq C_\lambda''\|f\|_p$. This proves that for every $\lambda\notin\mathbb Z$, the inverse operator $(D-\lambda I)^{-1}:L^p(\mathbb T)\to W^{1,p}(\mathbb T)$ is bounded, implying $\mathbb C\setminus\mathbb Z\subseteq \rho(D)$. 
Conversely, for every $k\in\mathbb Z$, the element $e_k(x)=e^{ikx}$ satisfies $De_k = k e_k$. Therefore, $\sigma(D)=\mathbb Z$ and $\rho(D)=\mathbb C\setminus\mathbb Z$, which is strictly non-empty.
To show that $D$ has a compact resolvent, choose $\lambda_0=\frac12 \in \rho(D)$. The resolvent $(D-\lambda_0 I)^{-1}$ maps $L^p(\mathbb T)$ boundedly into $W^{1,p}(\mathbb T)$. 
By the Rellich-Kondrachov theorem \cite[Theorem 6.3]{Adams}, the embedding 
$W^{1,p}(\mathbb T)\hookrightarrow L^p(\mathbb T)$ is 
is compact for every $1\le p<\infty$. 
Since the composition of a bounded operator with a compact embedding is  compact, we conclude that $(D-\lambda_0 I)^{-1}$ is compact.
Define the domain of the commutator derivation as $$\mathcal A_D = \{a\in C(\mathbb T):M_a\in\operatorname{Dom}(\delta_D)\}.$$ Let $\mathcal P$ denote the algebra of trigonometric polynomials
$$\mathcal P=\left\{a(x)=\sum_{|k|\leq N}a_k e^{ikx}: N\in\mathbb N,\ a_k\in\mathbb C \right\}.$$
For any $a\in\mathcal P$, we have $a\in C^\infty(\mathbb T)$. For any 
$\xi\in W^{1,p}(\mathbb T)$, we have
$$[D,M_a]\xi = D(a\xi)-aD\xi = -i(a'\xi+a\xi')+ia\xi'=-ia'\xi.$$
Thus, on $\operatorname{Dom}(D)$, we have 
\begin{equation}\label{eq44}
[D,M_a]=M_{-ia'}.
\end{equation} 
Since $a'\in C(\mathbb T)\subset L^\infty(\mathbb T)$, the operator $M_{-ia'}$ is bounded on $L^p(\mathbb T)$. This implies $\mathcal P\subseteq\mathcal A_D$. 
By the Stone-Weierstrass theorem, $\mathcal P$ is norm-dense in $C(\mathbb T)$. Additionally, $\mathcal A_D$ is a subalgebra because the derivation satisfies the Leibniz rule 
$$\delta_D(M_aM_b) = \delta_D(M_a)M_b+M_a\delta_D(M_b),$$ 
which remains bounded for $a,b\in\mathcal A_D$. Therefore, 
$(C(\mathbb T),L^p(\mathbb T),D)$ is a $L^p$-spectral triple.
	
Finally  we prove that  $(C(\mathbb T),L^p(\mathbb T),D)$ is strongly  dense-core $\beta$-Gevrey regular. Fix $\beta>0$ and $C>0$. It suffices to show that the Gevrey algebra $\mathcal A_{\beta,C}$ contains the norm-dense subalgebra $\mathcal P$. Let $a\in\mathcal P$ with $a(x)=\sum_{|k|\leq N}a_k e^{ikx}$. By Eq.\eqref{eq44}, the $n$-th iterated commutator is given by
$$\delta_D^n(M_a)=M_{(-i)^n a^{(n)}}.$$
Hence
$$\|\delta_D^n(M_a)\|_{\mathcal B(L^p(\mathbb T))}=\|a^{(n)}\|_\infty.$$
Since $a^{(n)}(x)=\sum_{|k|\leq N}(ik)^n a_k e^{ikx}$, we get
$$\|a^{(n)}\|_\infty\leq \sum_{|k|\leq N}|k|^n|a_k|\leq A_a N^n,$$
where $A_a=\sum\limits_{|k|\leq N}|a_k|$ is a constant independent of $n$.
Consequently, the Gevrey semi-norm 
$$L_{\beta,C}(a)=\sum_{n=1}^{\infty}\frac{\|\delta_D^n(M_a)\|_{\mathcal B(L^p(\mathbb T))}}{C^n(n!)^{1/\beta}}\leq A_a\sum_{n=1}^{\infty}\frac{(N/C)^n}{(n!)^{1/\beta}}.$$
this series converges because   
$$\frac{u_{n+1}}{u_n}=\frac{N/C}{(n+1)^{1/\beta}}\longrightarrow 0\quad\text{as} \ n\to\infty.$$
Thus, $L_{\beta,C}(a)<\infty$, meaning $a\in\mathcal A_{\beta,C}$. This proves that $\mathcal P\subseteq\mathcal A_{\beta,C}$. Because $\mathcal P$ is dense in $C(\mathbb T)$, the triple $(C(\mathbb T),L^p(\mathbb T),D)$ is  strongly  dense-core $\beta$-Gevrey regular for any arbitrary $\beta>0$.
\end{proof}

Motivated by Connes' classical construction of spectral triples \cite[page 208]{Rief4} from
proper length functions on discrete groups, we obtain the following
$L^p$-analogue.

\begin{theorem}\label{pro5}
Let $1\le p<\infty$, and let $G$ be a countable discrete group
equipped with a proper length function $\ell:G\to[0,\infty)$. Let
$D_\ell:\operatorname{Dom}(D_\ell)\subseteq \ell^p(G)\to \ell^p(G)$
be the multiplication operator
$(D_\ell\xi)(x)=\ell(x)\xi(x)$
with domain
$$\operatorname{Dom}(D_\ell)=\{\xi\in \ell^p(G):\ell\xi\in \ell^p(G)\}.$$
Then, for every $\beta>0$, $(F_r^p(G),\ell^p(G),D_\ell)$
is a  strongly  dense-core $\beta$-Gevrey regular  $L^p$-spectral triple with the derivation 
$\delta_{\ell}(\cdot)=[D_\ell, \cdot]$.
\end{theorem}

\begin{proof}
Let $\lambda_p:G\to\mathcal B(\ell^p(G))$ be the left regular
representation. $F_r^p(G)$ is a unital $L^p$-operator algebra with  unit 
$\lambda_p(e)=I_{\ell^p(G)}$, and the natural inclusion
$$\varphi:F_r^p(G)\hookrightarrow\mathcal B(\ell^p(G))$$
is a unital isometric representation. The space $C_c(G)$ of finitely supported functions is contained in
$\operatorname{Dom}(D_\ell)$ and is dense in $\ell^p(G)$. Hence
$D_\ell$ is densely defined. Moreover, as a multiplication operator by a real-valued function, $D_\ell$ is closed. 
We first verify the compact resolvent condition. Since $\ell\geq 0$, the operator
$D_\ell+I$ is bijective from $\operatorname{Dom}(D_\ell)$ onto $\ell^p(G)$, with inverse
$$(D_\ell+I)^{-1}=M_m,\qquad m(x)=\frac{1}{1+\ell(x)}.$$
Indeed, for $\eta\in\ell^p(G)$, setting
$$\xi(x)=\frac{\eta(x)}{1+\ell(x)}$$
gives $\xi\in\operatorname{Dom}(D_\ell)$ and $(D_\ell+I)\xi=\eta$.
Thus $-1\in\rho(D_\ell)$.
Since $\ell$ is proper, the sets $\{x\in G:\ell(x)\le N\}$ are finite. Let
$m_N=m\mathbf\chi_{\{\ell\le N\}}$. Then
the multiplication operator $M_{m_N}$ is finite rank and 
$$\|M_m-M_{m_N}\|=\sup_{\ell(x)>N}\frac{1}{1+\ell(x)}\le\frac{1}{1+N}\to 0.$$
Therefore $(D_\ell+I)^{-1}$ is compact.

Next we verify  the commutator condition holds on a dense
subalgebra. We canonically identify each group element $g \in G$ with the point mass $\delta_g \in \ell^1(G)$. Thus, for $g\in G$, the operator $\lambda_p(g)$ is given by
$$(\lambda_p(g)\xi)(x)=\xi(g^{-1}x)$$
for any $\xi \in \ell^p(G)$. If $\xi\in\operatorname{Dom}(D_\ell)$, using  the triangle inequality for the length function we get  $\ell(x) = \ell(g \cdot g^{-1}x) \le \ell(g) + \ell(g^{-1}x)$. Hence
$$\ell(x)|\xi(g^{-1}x)|\le\ell(g)|\xi(g^{-1}x)|+\ell(g^{-1}x)|\xi(g^{-1}x)|.$$
Taking the $\ell^p$-norm on both sides, we obtain	$$\|D_\ell\lambda_p(g)\xi\|_p = \|\ell\lambda_p(g)\xi\|_p\le\ell(g)\|\xi\|_p+\|\ell\xi\|_p<\infty.$$
Thus, $\lambda_p(g)\operatorname{Dom}(D_\ell)\subset\operatorname{Dom}(D_\ell)$. Moreover,  for $\xi\in\operatorname{Dom}(D_\ell)$, 
\begin{equation}\label{eqcl12}
([D_\ell,\lambda_p(g)]\xi)(x) = (\ell(x)-\ell(g^{-1}x))\xi(g^{-1}x).
\end{equation}
Since $|\ell(x)-\ell(g^{-1}x)| \le \ell(g)$, we get
$$\|[D_\ell,\lambda_p(g)]\xi\|_p \le \ell(g)\|\xi\|_p.$$
Therefore $[D_\ell,\lambda_p(g)]$  extends to a bounded operator on $\ell^p(G)$.
By linearity, for every $a\in\lambda_p(\mathbb C[G])$,
the operator $a$ preserves $\operatorname{Dom}(D_\ell)$ and
$[D_\ell,a]$ extends to a bounded operator. Hence
$$\lambda_p(\mathbb C[G])
\subseteq(F_r^p(G))_{D_\ell}=\{a\in F_r^p(G): a\in \operatorname{Dom}(\delta_{\ell})\}.$$
Since $\lambda_p(\mathbb C[G])$ is norm dense in $F_r^p(G)$, it follows
that $(F_r^p(G))_{D_\ell}$ is norm dense in $F_r^p(G)$.
$\operatorname{Dom}(\delta_\ell)$ is a subalgebra of
$\mathcal B(\ell^p(G))$ by the identity
$$[D_\ell,ST]=[D_\ell,S]T+S[D_\ell,T],$$
valid initially on $\operatorname{Dom}(D_\ell)$ and then by bounded
extension. Hence $(F_r^p(G))_{D_\ell}$ is a norm-dense subalgebra of
$F_r^p(G)$. Therefore $(F_r^p(G),\ell^p(G),D_\ell)$ is an $L^p$-spectral triple.

Finally we prove  the strongly dense-core $\beta$-Gevrey regularity. Since $\lambda_p(\mathbb C[G])$ is a norm-dense subalgebra of
$F_r^p(G)$, it suffices to prove that
$$\lambda_p(\mathbb C[G])\subseteq\bigcap_{C>0}\mathcal A_{\beta,C}.$$
For $g\in G$, put $h_g(x)=\ell(x)-\ell(g^{-1}x)$.
Then $|h_g(x)|\le\ell(g)$ for all $x\in G$. Hence $h_g\in \ell^\infty(G)$.
By equality \eqref{eqcl12},
$$\delta_\ell(\lambda_p(g))=M_{h_g}\lambda_p(g).$$
We claim that, for every $n\ge 1$,
$$\delta_\ell^n(\lambda_p(g))=M_{h_g^n}\lambda_p(g).$$
Indeed, let $T_{g,n}=M_{h_g^n}\lambda_p(g)$. Since $h_g^n$ is bounded
and $\lambda_p(g)$ preserves $\operatorname{Dom}(D_\ell)$, the operator
$T_{g,n}$ preserves $\operatorname{Dom}(D_\ell)$. Moreover, for
$\xi\in\operatorname{Dom}(D_\ell)$,
$$[D_\ell,T_{g,n}]\xi=M_{h_g^{n+1}}\lambda_p(g)\xi.$$
Since $h_g^{n+1}$ is bounded, this commutator extends to a bounded
operator on $\ell^p(G)$. Therefore $T_{g,n}\in\operatorname{Dom}(\delta_\ell)$,
and the claim follows by induction.
Consequently,
\begin{equation}\label{cheq10}
\|\delta_\ell^n(\lambda_p(g))\|=\|M_{h_g^n}\lambda_p(g)\|\le\|h_g^n\|_\infty\le\ell(g)^n.
\end{equation}
Now, let $f\in \mathbb C[G]$. We can write $\lambda_p(f)$ as a finite sum
$$\lambda_p(f)=\sum_{g\in\operatorname{supp}(f)}f(g)\lambda_p(g).$$
Since $\operatorname{Dom}(\delta_{\ell}^n)$ is a linear space, $\lambda_p(f)\in\operatorname{Dom}(\delta_{\ell}^n)$ for every $n\ge 1$, and we have
$$\delta_{\ell}^n(\lambda_p(f))=\sum_{g\in\operatorname{supp}(f)} f(g)\delta_{\ell}^n(\lambda_p(g)).$$
Thus by \ref{cheq10},
$$\|\delta_{\ell}^n(\lambda_p(f))\|\le\sum_{g\in\operatorname{supp}(f)} |f(g)|\ell(g)^n.$$
Let $R_f=\max\{\ell(g):g\in\operatorname{supp}(f)\}$. Then $\ell(g)^n \le R_f^n$ 
for all $g \in \operatorname{supp}(f)$, which implies
	$$\|\delta_{\ell}^n(\lambda_p(f))\|\le \|f\|_1 R_f^n.$$
Therefore, for any given constant $C>0$, we have
$$L_{\beta,C}(\lambda_p(f)) \le \|f\|_1 \sum_{n=1}^{\infty} \frac{(R_f/C)^n}{(n!)^{1/\beta}}.$$
The series converges by the ratio test, since
$$\frac{(R_f/C)^{n+1}/((n+1)!)^{1/\beta}}{(R_f/C)^n/(n!)^{1/\beta}}
=\frac{R_f/C}{(n+1)^{1/\beta}}\to 0.$$
Thus $L_{\beta,C}(\lambda_p(f))<\infty$ for every $C>0$.
Since the estimate holds for every $C>0$, we have
$$\lambda_p(\mathbb C[G])\subseteq\bigcap_{C>0}\mathcal A_{\beta,C}.$$
Moreover, $\lambda_p(\mathbb C[G])$ is a norm-dense subalgebra of
$F_r^p(G)$.  Therefore $L^p$-spectral triple
$(F_r^p(G),\ell^p(G),D_\ell)$ is strongly  dense-core
$\beta$-Gevrey regular.
\end{proof}

\section{Quantum Metrics on $L^p$-Group Algebras}

In this section, we investigate $p$-quantum compact metric spaces.
Motivated by Rieffel's theory of compact quantum metric spaces, we recall
several preliminary definitions, beginning with a general definition of
states on a unital Banach algebra \cite[Definition 2.6]{Blecher}.
\begin{defn}
Let $\mathcal{A}$ be a unital Banach algebra over $\mathbb{C}$ with unit $1_{\mathcal{A}}$. Let $\mathcal{A}^\prime$ denote the dual space of $\mathcal{A}$. The \emph{state space} of $\mathcal{A}$, denoted by $S(\mathcal{A})$, is defined by
$$\mathcal S(\mathcal A):=\{\omega\in \mathcal A^\prime: \ \|\omega\|=\omega(1_{\mathcal{A}})=1\}$$
An element $\omega \in S(\mathcal{A})$ is called a \emph{state} on $\mathcal{A}$.
\end{defn}

In analogy with Rieffel's construction, we use the Gevrey seminorm
associated with a strongly dense-core $\beta$-Gevrey regular
$L^p$-spectral triple to define a Monge--Kantorovich-type distance on the
state space. In the $L^p$-setting, the Gevrey seminorm plays the role of
a Lip-norm.
\begin{defn}
Let $p\in[1,\infty)$, and let $(\mathcal A,L^p(\mu),D)$ be a strongly dense-core
$\beta$-Gevrey regular $L^p$-spectral triple. Let $S(\mathcal A)$ denote
the state space of $\mathcal A$, and let $L_{\beta,C}$ be the associated
Gevrey seminorm. Define
$$\mathcal A_{\beta,C}=\{a\in\mathcal A^\infty: L_{\beta,C}(a)<\infty\}.$$
Assume that $\mathcal A_{\beta,C}$ is dense in $\mathcal A$ with respect
to the Banach algebra norm $\|\cdot\|_{\mathcal A}$. The
\emph{Monge-Kantorovich metric} associated with
$L_{\beta,C}$ is defined by
$$mk_{\beta,C}(\omega,\psi)=
\sup\left\{|\omega(a)-\psi(a)|:a\in\mathcal A_{\beta,C},\ L_{\beta,C}(a)\leq 1\right\}$$
for all $\omega,\psi\in S(\mathcal A)$.
\end{defn}

We now introduce a $p$-analogue for quantum compact metric spaces for our strongly  dense-core 
$\beta$-Gevrey regular $L^p$-spectral triple (Definition \ref{def35} ).
\begin{defn}\label{de43}
Let $(\mathcal A,L^p(\mu),D)$ be a strongly  dense-core $\beta$-Gevrey regular $L^p$-spectral triple. We say that
$(\mathcal A,L^p(\mu),D)$ is \emph{strongly $\beta$-metric} if, for every
$C>0$, Monge-Kantorovich metric $mk_{\beta,C}$ is a finite-valued
metric on $S(\mathcal A)$ and the topology induced by $mk_{\beta,C}$
coincides with the weak-$*$ topology
$\sigma(\mathcal A',\mathcal A)$ restricted to $S(\mathcal A)$.
In this case, the family $\{mk_{\beta,C}:C>0\}$ is called \emph{a strongly
$\beta$-Gevrey $p$-quantum compact metric} induced by $D$.
\end{defn}

M. A. Rieffel gives in \cite[Theorem~1.8]{Rief1} necessary and
sufficient conditions to define a metric $d_S$ on a set
$S\subseteq B'$ where $B$ is simply a normed vector space such that the
$d_S$-topology agrees with the weak-$*$ topology that $S$ inherits from
$B'$. We will use the sufficient conditions in \cite[Theorem~1.8]{Rief1}
to produce $p$-quantum compact metrics on $L^p$-spectral triple $(F_r^p(G),\ell^p(G),D_\ell)$. For the reader's convenience, we recall below M. A. Rieffel's setting,
which can also be found in \cite[page 14]{Del}.

The data consists of:
\begin{enumerate}
	\item[(R1)] A normed space $B$, with norm $\|\cdot\|_B$, over either
	$\mathbb C$ or $\mathbb R$.
	
	\item[(R2)] A subspace $\mathcal L$ of $B$, not necessarily closed.
	
	\item[(R3)] A seminorm $L$ on $\mathcal L$.
	
	\item[(R4)] A continuous, for $\|\cdot\|_B$, linear functional
	$\varphi$ on$$\mathcal N=\{a\in\mathcal L:L(a)=0\}$$
	with $\|\varphi\|=1$. In particular, we require
	$\mathcal N\neq\{0\}$.
	
	\item[(R5)] Let $$B'=\mathcal B(B,\mathbb C)$$
    and set
	$$S:=
	\{\omega\in B': \omega=\varphi \text{ on } \mathcal N,
	\text{ and } \|\omega\|=1\}.$$
\end{enumerate}

Thus $S$ is a norm-closed, bounded, convex subset of $B'$, and so is
weak-$*$ compact. We ask that $\mathcal L$ separate the points of $S$.
This means that given distinct $\omega,\psi\in S$ there is a
$b\in\mathcal L$ such that
$$\omega(b)\neq \psi(b).$$
With the data in $(R1)$--$(R5)$, we define a function
$$d_S:S\times S\to \mathbb R_{\geq 0}\cup\{\infty\}$$
by the formula
$$d_S(\omega,\psi)
:=\sup\{|\omega(b)-\psi(b)|: b\in\mathcal L,\ L(b)\le 1\}.$$

The assumed data guarantees that $d_S$ is an extended metric on $S$ that
separates the points of $S$. We will refer to the topology on $S$ defined
by $d_S$ as the ``$d_S$-topology''. Further, $L$ descends to an actual
norm on the quotient space $\mathcal L/\mathcal N$, but this quotient
space is also equipped with the quotient norm induced by $\|\cdot\|_B$.
The latter norm is the one we care about, which we will henceforward
denote by $\|\cdot\|_{B/\mathcal N}$.

The following is the forward direction of Theorem~1.8 in \cite{Rief1}.
\begin{theorem}\label{th54}
Let $B$, $\mathcal L$, $\mathcal N$, and $S$ be as in $(R1)$--$(R5)$.
Define $\mathcal L_1\subseteq \mathcal L$ by
$$\mathcal L_1:=\{b\in\mathcal L:L(b)\le 1\}.$$
If the image of $\mathcal L_1$ in $\mathcal L/\mathcal N$ is totally
bounded for $\|\cdot\|_{B/\mathcal N}$, then the $d_S$-topology on $S$
agrees with the weak-$*$ topology.
\end{theorem}

\bigskip
Next, we apply Rieffel's theorem to prove that, under 
assumption that $G$ satisfies $(GRD)_{\beta, p}$, the $L^p$-spectral triples
$(F_r^p(G),\ell^p(G),D_\ell)$ are strongly $\beta$-metric in the sense
of Definition \ref{de43}. We first establish the following lemma, which provides a bridge between Gevrey estimates and
subexponential tails and will be
used in the proof.

\begin{proposition}\label{pro55}
Let $C>0$, $c_\ell>0$, and $\beta>0$. For each integer $m \ge 1$, define
$$I_m = \inf_{n \ge 1} \frac{C^n (n!)^{1/\beta}}{(m c_\ell)^n}.$$
Then for any constant $k$ satisfying
$0 < k < \frac{1}{\beta} \left( \frac{c_\ell}{C} \right)^\beta,$
there exists a constant $A > 0$, depending only on $C$, $c_\ell$, $\beta$,
and $k$, such that for all integers $m \ge 1$, we have
$$I_m\le A \exp(-k m^\beta).$$
\end{proposition}
\begin{proof}
Set $x_m = \left( \frac{m c_\ell}{C} \right)^\beta.$
The infimum becomes
$$I_m = \inf_{n \ge 1} \left( \frac{n!}{x_m^n} \right)^{1/\beta}.$$
We divide the proof into two cases.

\emph{Case 1: $x_m\ge 1$.}
Let $n=\lfloor x_m \rfloor$. Then $x_m-1<n\le x_m$.
Since for all  $n\ge 1$, 
$$n!\le e\sqrt{n}\left(\frac{n}{e}\right)^n=e\sqrt{n} \, n^n e^{-n},$$
we have
$$\frac{n!}{x_m^n}\le\frac{e\sqrt{n}n^ne^{-n}}{x_m^n}=e\sqrt n
\left(\frac{n}{x_m}\right)^ne^{-n}.$$
Since $(n/x_m)^n\le 1$,  $\sqrt{n}\le\sqrt{x_m}$  and $e^{-n}<e^{-(x_m-1)}=e\cdot e^{-x_m},$
we get
$$I_m\le\left(e^2\sqrt{x_m}e^{-x_m}\right)^{1/\beta}=e^{2/\beta}x_m^{1/(2\beta)}\exp\left( -\frac{x_m}{\beta}\right).$$
Substituting $x_m=(m c_\ell / C)^\beta$ back into the inequality yields
$$I_m\le e^{2/\beta}\left(\frac{c_\ell}{C}\right)^{1/2} m^{1/2}\exp\left(-\frac{1}{\beta} \left(\frac{c_\ell}{C}\right)^\beta m^\beta\right).$$
Set $k_0 = \frac{1}{\beta} \left( \frac{c_\ell}{C} \right)^\beta$,
and  choose any $0 < k < k_0$. We have
$$I_m\le e^{2/\beta}\left(\frac{c_\ell}{C}\right)^{1/2} m^{1/2} e^{-(k_0-k)m^\beta}e^{-k m^\beta}.$$
Let $\alpha=k_0-k>0$ and Let 
$M_1=\sup_{t\ge 0} t^{1/2}e^{-\alpha t^\beta}<\infty.$
Hence, when $x_m\ge 1$, we have
$I_m\le A_1 e^{-k m^\beta}$,
where $A_1=e^{2/\beta} (c_\ell / C)^{1/2} M_1$. Thus $A_1$ depends only on $C$, $c_\ell$, $\beta$, and $k$, and is
independent of $m$.\\

\emph{Case 2:} $x_m<1$. 
Since $m$ is a positive integer, there are only finitely many values of $m$ that satisfy this condition.
For these specific values of $m$, we have
$$I_m\le\frac{C (1!)^{1/\beta}}{m c_\ell}=\frac{C}{m c_\ell}\le\frac{C}{c_\ell}.$$
Because the set $\{m \in \mathbb{N} : m <C/c_\ell\}$ is finite, there exists a constant $A_2>0$, depending only on $C$,
$c_\ell$, and $k$, such that
$$\frac{C}{c_\ell} \le A_2 e^{-k m^\beta}$$
for all such $m$. Therefore, we also have
$I_m \le A_2 e^{-k m^\beta}$,  whenever $x_m<1$. 
Taking $A=\max\{A_1,A_2\}$ 
we obtain
$I_m\le A e^{-k m^\beta}$,
for all integers $m\ge 1$. The constant $A$ depends only on $C$,
$c_\ell$, $\beta$, and $k$, and is independent of $m$. 
\end{proof}

The next theorem is the main result of this section.
\begin{theorem}\label{main}
Let $p\in[1,\infty)$, and let $G$ be a countable discrete group
equipped with a proper length function $\ell$. Let $F_r^p(G)$ denote the
reduced $L^p$-operator algebra of $G$, and define the unbounded operator
$D_\ell$ on $\ell^p(G)$ by $(D_\ell\xi)(g)=\ell(g)\xi(g)$.
Assume that $G$ satisfies $(GRD)_{\beta,p}$.
Then, for every $C>0$, the image of
$$\mathcal B_{\beta,C}=\{a\in\lambda_p(\mathbb C[G]):L_{\beta,C}(a)\le1\}$$
in $F_r^p(G)/\mathbb C1$ is totally bounded with respect to the quotient norm.
Consequently, the metric 
$$d_{\beta,C}(\omega,\psi)=\sup\{|\omega(a)-\psi(a)|:a\in\lambda_p(\mathbb C[G]),L_{\beta,C}(a)\le1\}$$
metrizes the weak-$*$ topology on $S(F_r^p(G))$.
\end{theorem}

\begin{proof}

For the sake of a clear exposition, we arrange our proof in several steps.

\medskip

\noindent {\bf Step 1.} $\ker L_{\beta, C}=\mathbb{C}1$.\\
\indent For $a=\lambda_p(f)\in\mathcal{L}$, the proof of Proposition \ref{pro5} shows Gevrey seminorm 
$$L_{\beta,C}(\lambda_p(f))<\infty.$$
Since each derivation $\delta_\ell^n$ is linear and the operator norm is absolutely homogeneous and subadditive, $L_{\beta, C}$ is a well-defined seminorm on $\lambda_p(\mathbb C[G])$.
Let $\delta_e$ be Dirac mass. Since $a\delta_e = f$ and $D_\ell\delta_e=0$, we have
$$(\delta_\ell(a)\delta_e)(g)=(D_\ell a\delta_e)(g)-(a D_\ell \delta_e)(g)=\ell(g)f(g).$$
By induction,  for every $n\ge 1$, we get
\begin{equation}\label{eq31}
(\delta_\ell^n(a)\delta_e)(g)=\ell(g)^n f(g).
\end{equation} 
Suppose  
$L_{\beta, C}(a)=\sum_{n=1}^{\infty}\frac{\left\|\delta_\ell^n(a) \right\|_{B(\ell^p(G))}}{C^n(n!)^{1/\beta}}= 0$. 
Since all terms in the defining series are non-negative, we must have $\|\delta_\ell^n(a)\|_{B(\ell^p(G))}=0$, implying $\delta^n_\ell(a)=0$. By \eqref{eq31}, 
$$0=(\delta^n_\ell(a)\delta_e)(g)=\ell(g)^nf(g)\quad \text{for all}\  g \in G. $$
Since $\ell(g)=0$ if and only if $g=e$, we have $\ell(g)>0$ for all
$g\ne e$. This forces $f(g)=0$ for all $g\neq e$. Thus, 
$f=f(e)\delta_e$. Hence $a=\lambda_p(f)=f(e)1\in\mathbb{C}1$. Conversely, any scalar multiple of the identity commutes with $D_\ell$, so $L_{\beta, C}(c1)=0$ for all $c \in \mathbb{C}$. Therefore,  $\ker L_{\beta, C}=\mathbb{C}1$.

\medskip 

\noindent {\bf Step 2.} Uniform tail estimates.\\
\indent Let $$\mathcal B_{\beta,C}^0=\left\{a=\lambda_p(f)\in\mathcal L:
L_{\beta,C}(a)\le 1,\ f(e)=0\right\}.$$
For $a=\lambda_p(f)\in\mathcal B_{\beta,C}^0$.
The condition $L_{\beta, C}(a)\le 1$ implies that for every $n\ge 1$, $\|\delta_\ell^n(a)\| \le C^n(n!)^{1/\beta}$. Taking the $\ell^p$-norm in \eqref{eq31}, we have
\begin{equation}\label{eq35}
\|f\ell^n\|_{\ell^p(G)}=\|\delta_\ell^n(a)\delta_e\|_{\ell^p(G)}\le \|\delta_\ell^n(a)\|_{B(\ell^p(G))} \le C^n(n!)^{1/\beta}. 
\end{equation}
In particular, for $n=1$, since $\ell$ is proper and $\ell(g)=0$ if and only if $g=e$, we have $$c_\ell:=\inf_{g\ne e}\ell(g)>0.$$ 
As $f(e)=0$, it follows that
\begin{equation}\label{eq33}
\|f\|_{\ell^p(G)}\le c_\ell^{-1}\|f\ell\|_{\ell^p(G)}\le c_\ell^{-1}C.
\end{equation}
For each integer $m\ge 1$, set $A_m=\{g\in G: m c_\ell\le\ell(g)<(m+1)c_\ell\}$, and define  $f_m=f\cdot\chi_{A_m}$. Since $f(e)=0$, we have $f=\sum_{m\ge 1}f_m$, where the sum is finite since  $f\in\mathbb{C}[G]$. Thus for any $n\ge1$, by \eqref{eq35}, we obtain 
$$(m c_\ell)^n\|f_m\|_{\ell^p(G)}\le\|f_m\ell^n\|_{\ell^p(G)}\le\|f \ell^n\|_{\ell^p(G)} \le C^n (n!)^{1/\beta},$$
that is
$$\|f_m\|_{\ell^p(G)}\le\frac{C^n (n!)^{1/\beta}}{(m c_\ell)^n}.$$
Minimizing over $n$, we obtain $\|f_m\|_{\ell^p(G)} \le \inf_{n \ge 1} \frac{C^n (n!)^{1/\beta}}{(m c_\ell)^n}$. 
By Lemma~\ref{pro55}, for any 
$0<k<\frac{1}{\beta}\left(\frac{c_\ell}{C}\right)^\beta$,
there exists a constant $A>0$, depending only on $C,c_\ell,\beta$ and $k$ and independent of the choice of $a\in B_{\beta,C}^0$, such that the decomposition terms
$f_m$ of $a$ satisfy
\begin{equation}\label{eq54}
	\|f_m\|_{\ell^p(G)}\le A\exp(-k m^\beta)
\end{equation}
for all $m\ge 1$.
We now use the assumption that $G$ satisfies $(GRD)_{\beta,p}$.
Since
$$\operatorname{supp}(f_m)\subseteq B_{(m+1)c_\ell},$$
we obtain
$$\|\lambda_p(f_m)\|_{\mathcal B(\ell^p(G))}\le F((m+1)c_\ell)\|f_m\|_{\ell^p(G)}, $$
and $$\log F((m+1)c_\ell)=o(((m+1)c_\ell)^\beta).$$
By \eqref{eq54},  it follows that
\begin{equation}\label{eq55}
\|\lambda_p(f_m)\|_{\mathcal{B}(\ell^p(G))}\le A F((m+1)c_\ell)\exp(-km^\beta). 
\end{equation}
As $((m+1)c_\ell)^\beta=O(m^\beta)$, we have
$\log F((m+1)c_\ell)=o(m^\beta)$.
Hence, for any  $0<\varepsilon<k$, 
there exists an integer $M\ge 1$,  which is independent of $a$, such that
\begin{equation}\label{eq56}
F((m+1)c_\ell)\le\exp(\varepsilon m^\beta)
\end{equation}
for all $m\ge M$.
We can choose $R$ large enough such that 
$R\ge M$ and consider the tail
$$a_{>R}=\sum_{m>R}\lambda_p(f_m).$$
Thus, by \eqref{eq55} and \eqref{eq56},  we deduce
\begin{align*}
\|a_{>R}\|_{F_r^p(G)}
&\le\sum_{m>R}\|\lambda_p(f_m)\|_{\mathcal B(\ell^p(G))}\le A\sum_{m>R}
F((m+1)c_\ell)\exp(-k m^\beta)\\
&\le A\sum_{m>R}\exp(\varepsilon m^\beta)\exp(-k m^\beta)=A\sum_{m>R}
\exp(-(k-\varepsilon)m^\beta).
\end{align*}
Let
$$S_R=A\sum_{m>R}\exp(-(k-\varepsilon)m^\beta).$$
Since $k-\varepsilon>0$, the series
$\sum_{m=1}^{\infty}\exp(-(k-\varepsilon)m^\beta)$
converges. Consequently, $\lim_{R\to\infty}S_R=0$.
Because the bound $S_R$ is independent of $a\in B_{\beta,C}^0$, we conclude 
\begin{equation} \label{eq36}
\lim_{R \to \infty} \|a_{>R}\|_{F_r^p(G)} = 0.
\end{equation}
and the convergence is uniform over $\mathcal B_{\beta,C}^0$.

 \medskip 
 
 \noindent {\bf Step 3.} The set of truncated operators is totally bounded.\\
 \indent 
In this step, we show that, for each fixed $R\ge1$, the set of
truncations
$$\{a_{\le R}:a\in\mathcal B_{\beta,C}^0\}$$
is totally bounded in $F_r^p(G)$.  Since  
$$a_{\le R}=a-a_{>R}=\sum_{1 \le m\le R}\lambda_p(f_m),$$
by the linearity of the regular representation, we can write 
$a_{\le R}=\lambda_p(f_{\le R})$, where $f_{\le R}=\sum\limits_{1\le m\le R}f_m$. Thus, since $f(e)=0$, we have 
$$\text{supp}(f_{\le R})\subset\bigcup_{m=1}^R A_m=\{g\in G: c_\ell\le\ell(g)< (R+1)c_\ell\}.$$ 
Consequently, $a_{\le R}$ is a finite linear combination of the group operators $\lambda_p(g)$ for $g$ in this specific range. 
This  shows that all truncations $a_{\le R}$ belong to the subspace 
$$V_R=\operatorname{span}\{\lambda_p(g): c_\ell\le \ell(g)<(R+1)c_\ell\}.$$ 
Since  the length function $\ell$ is proper, the ball $B_{(R+1)c_\ell}$ is finite. Hence $V_R$ is a finite-dimensional space. Moreover, by \eqref{eq33}, we get 
$$\|f_{\le R}\|_{\ell^p(G)}\le\|f\|_{\ell^p(G)}\le c_\ell^{-1}C.$$ 
Since $G$ satisfies $(GRD)_{\beta,p}$, we obtain
$$ \|a_{\le R}\|_{F_r^p(G)}\le F((R+1)c_\ell)\|f_{\le R}\|_{\ell^p(G)}\le F((R+1)c_\ell)c_\ell^{-1}C. $$
Consequently, the set of all such truncations $\{a_{\le R}\}$ is bounded in the finite dimensional space $V_R$,  and thus it has a compact closure by the Heine-Borel theorem, which immediately implies that $\{a_{\le R}\}$ is totally bounded.

\medskip 

\noindent {\bf Step 4.} Total boundedness of the image of $\mathcal B_{\beta,C}$.\\
\indent
Fix $\eta>0$. By \eqref{eq36}, we can choose $R$ sufficiently large such that
$$\|a-a_{\le R}\|_{F_r^p(G)}<\eta$$
for all $a\in \mathcal B_{\beta,C}^0$.
For such $R$, by Step 3, the set
$\{a_{\le R}:a\in\mathcal B_{\beta,C}^0\}$
is totally bounded in $V_R$, where
$$V_R=\operatorname{span}\{\lambda_p(g): c_\ell\le\ell(g)<(R+1)c_\ell\}.$$
Hence there exist finitely many elements $x_1,\dots,x_N\in V_R$ such that for every 
$a\in \mathcal B_{\beta,C}^0$, there is an $i$ with
$$\|a_{\le R}-x_i\|_{F_r^p(G)}<\eta.$$
It follows that
$$\|a-x_i\|_{F_r^p(G)}\le\|a-a_{\le R}\|_{F_r^p(G)}+\|a_{\le R}-x_i\|_{F_r^p(G)}<2\eta.$$
Therefore $\mathcal B_{\beta,C}^0$ is totally bounded in $F_r^p(G)$.
Let $\pi:F_r^p(G)\to F_r^p(G)/\mathbb C1$
be the quotient map, and let
$$\mathcal B_{\beta,C}=\{a\in\mathcal L:L_{\beta,C}(a)\le1\}.$$
For any $a=\lambda_p(f)\in\mathcal B_{\beta,C}$, set
$$a_0=a-f(e)1=\lambda_p(f-f(e)\delta_e).$$
Then $\pi(a_0)=\pi(a)$ and $a_0(e)=0$. Moreover, since
$\delta_\ell^n(1)=0$ for all $n\ge1$,
$$L_{\beta,C}(a_0)=L_{\beta,C}(a)\le1.$$
Thus $a_0\in\mathcal B_{\beta,C}^0$, and hence
$\pi(\mathcal B_{\beta,C})=\pi(\mathcal B_{\beta,C}^0)$.
Since $\mathcal B_{\beta,C}^0$ is totally bounded in $F_r^p(G)$ and
$\pi$ is contractive, $\pi(\mathcal B_{\beta,C})$ is totally bounded in
$F_r^p(G)/\mathbb C1$.

\medskip 

\noindent {\bf Step 5.} Apply the Rieffel criterion\\
\indent 
Let us set $B=F_r^p(G)$, the dense subspace to be $\mathcal{L}=\lambda_p(\mathbb C[G])$, and the seminorm to be $L=L_{\beta,C}$. 
By Step 1, the null space of $L_{\beta,C}$ is precisely
$$\mathcal{N}=\ker L_{\beta,C}=\mathbb{C}1.$$
We define a linear functional $\varphi: \mathcal{N}\to\mathbb{C}$ by $\varphi(c1)=c$. 
Since $1=\lambda_p(\delta_e)$ is the identity operator on $\ell^p(G)$, we have $\|1\|_{F_r^p(G)}=1$, which immediately implies that $\|\varphi\|=1$. Thus, the corresponding state space is given by
$$S_p(G)=\left\{\omega\in F_r^p(G)^*:\omega(1)=1,\ \|\omega\|=1 \right\}.$$
Next, we verify that $\mathcal L$ separates the points of $S_p(G)$. Indeed, if 
$\omega(a)=\psi(a)$ for all $a\in\mathcal L$ with $\omega, \psi\in S_p(G)$, the norm-density of $\mathcal L$ in $F_r^p(G)$ and the continuity of the functionals immediately imply that $\omega = \psi$ on the entire space $F_r^p(G)$. Thus, $\mathcal L$ separates the points of $S_p(G)$. By Step~4, the image of
$$\mathcal B_{\beta,C}=\{a\in\mathcal L:L_{\beta,C}(a)\le 1\}$$
in the quotient space $\mathcal L/\mathcal N$ is totally bounded with
respect to the quotient norm induced by the norm of $F_r^p(G)$.  Hence there exists a constant $M_C>0$ such that
$$\|a+\mathcal N\|_{F_r^p(G)/\mathcal N}\le M_C, $$ 
for every $a\in\mathcal B_{\beta,C}$.
Since $\mathcal N=\mathbb C1$, there exists a scalar $\lambda_a\in\mathbb C$
such that
$$\|a-\lambda_a1\|_{F_r^p(G)}\le M_C+1.$$
Let $\omega,\psi\in S_p(G)$. Since 
$$\omega(a)-\psi(a)=\omega(a-\lambda_a1)-\psi(a-\lambda_a1),$$
we get 
\begin{align*}
|\omega(a)-\psi(a)|
&\le|\omega(a-\lambda_a1)|+|\psi(a-\lambda_a1)|\\
&\le(\|\omega\|+\|\psi\|)\|a-\lambda_a1\|_{F_r^p(G)}\le 2(M_C+1).
\end{align*}
Taking the supremum  gives
$$d_{\beta,C}(\omega,\psi)\le2(M_C+1)<\infty.$$
Thus $d_{\beta,C}$ is finite-valued on
$S_p(G)\times S_p(G)$.
Hence all the hypotheses of Theorem \ref{th54} are satisfied. Therefore,
the metric
$$d_{\beta,C}(\omega,\psi)=\sup\{|\omega(a)-\psi(a)|:
a\in\mathcal L,\ L_{\beta,C}(a)\le1\}$$
induces on $S_p(G)$ precisely the weak-$*$ topology inherited from
$F_r^p(G)^*$.
\end{proof}

In terms of Definition \ref{de43} we have shown the following result.

\begin{corollary}\label{maincro}
Let $p\in[1,\infty)$, and let $G$ be a countable discrete group 
equipped with a proper length function $\ell$. Suppose that $G$ satisfies
$(GRD)_{\beta,p}$ for some $0<\beta\le 1$. Then strongly dense-core $\beta$-Gevrey regular
$L^p$-spectral triple$(F_r^p(G),\ell^p(G),D_\ell)$ is a  strongly $\beta$-metric.
\end{corollary}

\section{Example}
In this section we will give some examples of groups with $(GRD)_{\beta,p}$ for some $0<\beta\le 1$ and $0\le p<\infty$. Hence we will obtain compact quantum metric space structures on the resulting  reduced $L^p$-group algebras.  These include groups of polynomial growth, groups of intermediate growth such as the Grigorchuk group, more general groups of subexponential growth, and groups with the usual rapid decay property.

Recall that the first Grigorchuk group was introduced by Grigorchuk as the first example of a finitely generated group of intermediate growth \cite{Grig}. In light of its specific growth properties, the following proposition establishes that the compact quantum metric construction can be extended to groups of intermediate growth, and that its associated $L^p$-spectral triple is strongly $\beta$-metric.

\begin{proposition}
Let $\mathfrak G$ denote the first Grigorchuk group, equipped with the word length $\ell$ associated with its standard finite generating set. Then the $L^p$-spectral triple $(F_r^p(\mathfrak G),\ell^p(\mathfrak G),D_\ell)$ is strongly $\beta$-metric for every $p\in[1,\infty)$ and every $\alpha<\beta\le 1$.
\end{proposition}
\begin{proof}
According to Bartholdi's upper bound for the first Grigorchuk group \cite{Barth}, the volume of the ball of radius $R$ satisfies
	$$|B_R|\le C\exp(cR^\alpha) $$
for some constants $C, c > 0$ and $\alpha \approx 0.767$. 
By Remark \ref{re34}, this implies that $\mathfrak G$ satisfies $(GRD)_{\beta,p}$ with respect to $\ell$ for every $p\in[1,\infty)$ and every $\alpha<\beta\le 1$. Consequently, Corollary \ref{maincro} implies that the spectral triple 
$(F_r^p(\mathfrak G),\ell^p(\mathfrak G),D_\ell)$ is strongly $\beta$-metric for all such $p$ and $\beta$.
\end{proof}

Recall  that the length function $\ell$ has bounded doubling (see \cite[Definition 1.1]{ChriRi} and \cite[Definition 4.1]{Del}), namely there exists a constant
$C_\ell<\infty$ such that
$$|B_{2R}|\le C_\ell|B_R|\qquad (R\ge1).$$ 
In \cite[Theorem 4.6]{Del} the authors showed  in the bounded-doubling case, $L^p$-spectral triple $(F_r^p(\mathfrak G),\ell^p(\mathfrak G),D_\ell)$ is metric. The following proposition shows that the spectral triple is, in fact,
strongly $\beta$-metric.

\begin{proposition}
Let $G$ be a countable discrete group equipped with a proper length function $\ell$. Suppose that $\ell$ has bounded doubling. Then, the $L^p$-spectral triple 
$(F_r^p(G), \ell^p(G), D_\ell)$ is strongly $\beta$-metric for every 
$p\in [1,\infty)$ and every $0<\beta\le 1$.
\end{proposition}

\begin{proof}
By \cite[Proposition 1.2]{ChriRi}, if $\ell$ has bounded doubling, then it exhibits polynomial volume growth. Hence, 
$\log |B_R|=O(\log R)=o(R^\beta)$ for every $\beta>0$. 
Proposition \ref{pro21} then implies that $G$ satisfies $(GRD)_{\beta,p}$ with respect to $\ell$ for every $p \in [1,\infty)$ and every $\beta > 0$. Consequently, by Corollary \ref{maincro}, the spectral triple $(F_r^p(G), \ell^p(G), D_\ell)$ is strongly $\beta$-metric for every $p\in[1,\infty)$ and every 
$0<\beta\le 1$.
\end{proof}

\begin{example}
Let $G$ be a finitely generated group of subexponential growth, equipped
with a word length $\ell$,  Proposition \ref{pro21} implies  $G$ satisfies
$(GRD)_{1,p}$ for every $p\in[1,\infty)$.
Consequently, Corollary \ref{maincro} implies
$(F_r^p(G),\ell^p(G),D_\ell)$ is strongly $1$-metric for every $p\in[1,\infty)$.
\end{example}

\begin{example}
Let $G$ be a countable discrete group equipped with a length function $\ell$.
Assume that $G$ has property subexponential decay property (see \cite{Sri}) with respect to $\ell$. Remark \ref{re34}
implies $G$ satisfies $(GRD)_{1,2}$. Consequently, Corollary \ref{maincro} implies
$(C_r^*(G),\ell^2(G),D_\ell)$ is strongly $1$-metric.
\end{example}

\begin{example}
Let $G$ be a countable discrete group equipped with a length function $\ell$. Assume that $G$ satisfies the usual rapid decay property with respect
to $\ell$ in the Hilbert space case $p=2$. 
Proposition \ref{pro21}  implies that $G$ satisfies $(GRD)_{\beta,2}$. Consequently, by
Corollary \ref{maincro}, $(F_r^2(G),\ell^2(G),D_\ell)$
is strongly $\beta$-metric for every $0<\beta\le 1$. In this case we have $F_r^2(G)=C_r^*(G)$.
Hence, the compact quantum metric structures obtained here are
consistent with earlier constructions on reduced group $C^*$-algebras due to
Antonescu-Christensen \cite{Anto}, Ozawa-Rieffel \cite{Ozawa}, and
Christensen-Rieffel \cite{ChriRi}.
\end{example}

\subsection*{Acknowledgement} The first author was supported by NSFC (No. 12061018). The second author was supported by NSFC (No. 12571135, 12171156), Key Laboratory of MEA (Ministry of Education), the Science and Technology Commission of Shanghai (No. 22DZ2229014), and Shanghai Key Laboratory of PMMP, East China Normal University.

\end{document}